%% file: Brenner_Strong_Solution_Existence.tex
\documentclass[11pt]{amsart}

\input{newmacros}

\title[   ]{Global Existence of Classical Solutions to Brenner--Navier--Stokes--Fourier System for Large Data}
\author[Eo]{Saehoon Eo}
\address[Saehoon Eo]
{ Department of Mathematics, \newline
Stanford University \\
CA 94305, USA}
\email{eosehoon@stanford.edu}

\author[Eun]{Namhyun Eun}
\address[Namhyun Eun]
{ School of Mathematics, \newline
Korea Institute for Advanced Study (KIAS) \\
Seoul 02455, Republic of Korea}
\email{namhyuneun@kias.re.kr}

\author[Kang]{Moon-Jin Kang}
\address[Moon-Jin Kang]
{ Department of Mathematical Sciences, \newline
Korea Advanced Institute of
Science and Technology \\
Daejeon 34141, Republic of Korea}
\email{moonjinkang@kaist.ac.kr}
\newtheorem{theorem}{Theorem}[section]

\theoremstyle{remark}
\newtheorem{remark}{Remark}[section]

\theoremstyle{plain}

\newtheorem{theo}{Theorem}[section]
\newtheorem{prop}[theo]{Proposition}
\newtheorem{cor}[theo]{Corollary}
\newtheorem{lem}[theo]{Lemma}

\numberwithin{equation}{section}

\begin{document}
\bibliographystyle{plain}

\date{\today}



\thanks{\textbf{Acknowledgment.}
The authors thank Professor Alexis F. Vasseur for his valuable comments.
This work was supported by Samsung Science and Technology Foundation under Project Number SSTF-BA2102-01.}

\begin{abstract}
We study the 1D Brenner--Navier--Stokes--Fourier (BNSF) system, proposed as a refinement of the classical Navier--Stokes--Fourier model through the introduction of the volume velocity, distinct from the mass velocity describing convective transport.
When formulated in the Lagrangian mass coordinates with the volume velocity, the discrepancy between the two velocities induces a dissipative structure in the mass conservation law.

We prove the global existence of classical solutions for arbitrarily large initial data.
More precisely, for initial data in $H^k(\mathbb{R})$ with $k\ge3$, with the specific volume and absolute temperature initially bounded away from zero, we construct global-in-time solutions that remain in the same regularity class.
Our result accommodates arbitrarily large initial data.
A major difficulty is to establish lower and upper bounds for the specific volume \(v\).
The additional dissipation yields an $L_t^2 L_x^2$ bound for $v_x$, which is further improved to an $L_t^\infty L_x^\infty$ bound of $v$ and $1/v$ via the parabolic De Giorgi method.
We also apply the maximum principle to obtain a positive lower bound for the absolute temperature.

\end{abstract}
\maketitle \centerline{\date}

\tableofcontents

\section{Introduction}
This paper examines the one-dimensional Brenner--Navier--Stokes--Fourier (BNSF) system, which describes the motion of the viscous and heat-conducting fluids.
In the Lagrangian mass coordinates, the system is written as
\begin{equation}\label{bnsf0}
\begin{cases}
v_t - (u_m)_x = 0, \\
(u_v)_t + p(v, \th)_x = \m \big(\frac{(u_v)_x}{v}\big)_x, \\
E_t + (p u_v)_x = \k \big(\frac{\th_x}{v}\big)_x + \m\big(\frac{u_v(u_v)_x}{v}\big)_x,
\end{cases}
\end{equation}
where \(v,u_m,u_v,\th,p(v,\th),E\) respectively represent the specific volume, mass velocity, volume velocity, absolute temperature, pressure law, and the total energy.
The notion of the mass and volume velocity will be introduced shortly.
We consider the ideal polytropic gas, where the pressure law and the total energy are given by
\begin{equation}\label{pressure}
p(v,\th) = \frac{R\th}{v}, \qquad
E=\frac{R\th}{\g-1}+\frac{(u_v)^2}{2},
\end{equation}
where \(R>0\) denotes the gas constant and \(\g>1\) the adiabatic constant.
Moreover, \(\m\) and \(\k\) are positive constants representing the viscosity and heat conductivity coefficients, respectively.

\vspace{2mm}
The BNSF system was first introduced in a series of papers \cite{Brenner-1,Brenner-2,Brenner-3} to address certain limitations of the classical Navier--Stokes--Fourier (NSF) system, which fails to capture the behavior of fluids under extreme conditions such as rarefied gases, shock waves, and gaseous flows through micro-channel.
Numerous studies have pointed out shortcomings of the NSF equations and suggested various modifications (see for instance \cite{ABS01,CD07,DRM08,DDC10,DSD09,GR07,HHBZ95,RGL03}), and the BNSF system offers one of the most systematic and well-developed frameworks in the literature.

\vspace{2mm}
Brenner's framework is rooted in the bi-velocity theory, which postulates the presence of two fundamentally distinct velocity fields in fluid description: the mass velocity \(u_m\), coinciding with the conventional fluid velocity in classical formulations, and the volume velocity \(u_v\).
In \cite{Brenner-3}, Brenner argued that the prevailing body of continuum fluid literature tacitly identifies these two velocities, an assumption that has largely escaped critical scrutiny.
Contrary to this convention, Brenner emphasized that the two velocities \(u_m\) and \(u_v\) are generally non-equivalent, with their inconsistency becoming increasingly pronounced in regimes characterized by large density (or specific volume) gradients.
More precisely, in the Lagrangian mass coordinates, the constitutive relation linking \(u_m\) and \(u_v\), proposed by Brenner \cite{Brenner-3}, is given by
\begin{equation}\label{umuv}
u_m = u_v + \frac{\k}{c_p}\left(\frac{v_x}{v}\right),
\end{equation}
where \(c_p\) denotes the specific heat at constant pressure.
This expression indicates that the discrepancy between the two velocity fields scales with \(v_x\), implying that the notion of the volume velocity becomes negligible in incompressible flows with constant density.

\vspace{2mm}
In a series of papers \cite{Brenner-1,Brenner-2,Brenner-3}, Brenner proposed the constitutive relation \eqref{umuv} on the basis of \"Otinger's theoretical framework and substantiated it through multiple lines of reasoning, including its consistency with Burnett's solutions of the Boltzmann equation and agreement with experimental observations.
As asserted in \cite{Brenner-3}, the relation \eqref{umuv} is also compatible with the principles of linear irreversible thermodynamics.
Notably, for flows exhibiting nontrivial density gradients, Brenner argued that the no-slip boundary condition should be formulated in terms of the volume velocity, which yields improved concordance with empirical data.

\vspace{2mm}
The introduction of volume velocity may be motivated from a heuristic standpoint.
When a gas undergoes deformation under internal pressure, mechanical work is performed, the quantification of which necessitates a notion of displacement.
However, pressure is neither generated by nor acts upon isolated particles, but rather on ensembles of particles whose boundaries are inherently diffuse and ill-defined.
This renders a rigorous definition of displacement at the microscopic level problematic.
Such ambiguity reflects the intrinsically statistical nature of volume, which cannot be attributed to individual particles but instead pertains to particle aggregates.
From this perspective, the concept of volume velocity emerges as a necessary kinematic descriptor for the motion of such aggregates, complementing the classical velocity notion.
In essence, while the mass velocity \(u_m\) governs mass transport and convection, the volume velocity \(u_v\) is introduced to account for momentum transfer, energy transport, mechanical work, and viscous effects.

\vspace{2mm}
Despite the comprehensive formulation of Brenner's theory, rigorous mathematical investigations of the associated BNSF system remain limited.
Feireisl and Vasseur \cite{FV-Brenner10} showed the global existence of weak solutions for the multidimensional initial-boundary value problem.
Moreover, the BNSF framework has been employed in \cite{BF-Brenner18} to construct measure-valued solutions to the Euler system, as well as in \cite{FLM-Brenner20} in the development of finite volume schemes for the Euler equations.
More recently, Eo--Eun--Kang--Oh \cite{EEKO-JDE25} established the existence of traveling wave solutions---commonly referred to as viscous shocks---for the one-dimensional BNSF system, together with several quantitative properties of these solutions.
This was subsequently extended in \cite{EE-CMS26} to the case of general \(\Ccal^2\) temperature-dependent transport coefficients.
Furthermore, the nonlinear stability of such viscous shocks under arbitrarily large perturbations was demonstrated in \cite{EEK-BNSF}.
Notably, the absence of smallness assumptions on the perturbations enables the derivation of uniform estimates, which in turn provide a rigorous justification of the vanishing dissipation limit toward the Euler system.
Motivated by the aforementioned developments, the main objective of the current paper is to establish the global existence of classical solutions to the BNSF system.

\vspace{2mm}
By eliminating the mass velocity via the constitutive relation \eqref{umuv}, the BNSF system \eqref{bnsf0} can be reformulated as follows:
\begin{equation} \label{bnsfsys}
\begin{cases}
v_t - u_x = \t\big(\frac{v_x}{v}\big)_x, \\
u_t + p(v, \th)_x = \m\big(\frac{u_x}{v}\big)_x, \\
\frac{R}{\g-1}\th_t + p(v, \th)u_x = \k\big(\frac{\th_x}{v}\big)_x + \m\frac{(u_x)^2}{v},
\end{cases}
\end{equation}
where \(u\) now represent the volume velocity, and \(\t\) denotes the Brenner coefficient defined by \(\k/c_p\).
We remark that the distinction between the two velocities yields a dissipative structure in the mass conservation law, in contrast to the classical NSF system, which takes the following form:
\begin{equation} \label{nsf}
\begin{cases}
v_t - u_x = 0, \\
u_t + p(v, \th)_x = \m\big(\frac{u_x}{v}\big)_x, \\
\frac{R}{\g-1}\th_t + p(v, \th)u_x = \k\big(\frac{\th_x}{v}\big)_x + \m\frac{(u_x)^2}{v}.
\end{cases}
\end{equation}

\subsection{Main Result}
In this paper, we establish the global existence of classical solutions to the BNSF system.
The main result is stated as follows.

\begin{theorem}[Global Existence]\label{thm:main}
Let \(R,\g,\t,\m,\k\) be any positive constants with \(\g>1\) and \(k\ge 3\) be an integer.
Then, for an initial data \((v_0,u_0,\th_0)\) satisfying
\[
\norm{v_0-1}_{H^k(\RR)} + \norm{u_0}_{H^k(\RR)} + \norm{\th_0-1}_{H^k(\RR)} < \infty
\]
and
\[
\inf_{x \in \RR} (v_0(x), \th_0(x)) > 0,
\]
there exists a unique classical solution \((v, u, \th)\) to the system \eqref{bnsfsys} with \eqref{pressure} on \(\RR^+ \times \RR\) such that for every \(T>0\),
\[
(v-1, u, \th-1)\in \big(C([0,T]; H^k(\RR)) \cap L^2(0, T; H^{k+1}(\RR))\big)^3.
\]
Moreover, for every \(T>0\), there exist constants \(\underline{\ell}(T)\) and \(\overline{\ell}(T)\) such that 
\[
0 < \underline{\ell}(T) \le v(t,x), \th(t,x) \le \overline{\ell}(T), \qquad \forall (t,x)\in(0,T)\times\RR.
\]
\end{theorem}

We provide a remark regarding the main theorem.
\begin{remark}
(1) No restriction is imposed on the size of the initial data.
In particular, global classical solutions exist for arbitrarily large initial data, provided that the initial specific volume and temperature are uniformly positive.
(2) If the specific volume and the absolute temperature are initially bounded away from zero, then they remain strictly positive for all later times.
Establishing this non-vanishing property constitutes the main difficulty of the analysis: the lower (and upper) bound for the specific volume is obtained via the De Giorgi method, while that for the temperature follows from the maximum principle.
\end{remark}

In what follows, we review the existing literature on the existence theory for the classical system.
We begin with the barotropic Navier--Stokes (NS) equations, which possess a simpler yet still fundamental structure.
Prior to the 2000s, a number of seminal works established the existence of weak solutions for a broad class of initial data, including discontinuities and vacuum; see \cite{Hoff-87,Hoff-98,Serre-86,Shelukhin-83,Shelukhin-84}.
A major development in the 2000s was the discovery of the Bresch--Desjardins (BD) entropy structure \cite{BD-CMP03,BD-JMPA06,BDL-CPDE03}, which has since played a pivotal role in the analysis of compressible fluid models.
Based on this structure, Mellet--Vasseur \cite{MV-SIMA09} proved the global existence of strong solutions to the isentropic NS system.
A key ingredient in this work---and likewise in the present study---is the derivation of uniform upper and lower bounds for the density \(\r=1/v\) (or equivalently, the specific volume).
Their approach can be summarized as follows.
An energy estimate based on the relative entropy yields an \(L^\infty_t L^2_x\) bound for \(\sqrt{\r} u\).
In addition, exploiting the BD entropy structure provides a further estimate which, somewhat remarkably, controls an \(L^\infty_t L^2_x\) bound for \(\sqrt{\r} u + \r_x/\r^2\).
As a consequence, one obtains an \(L^\infty_t L^2_x\) bound for \(\r_x/\r^2\), which in turn enables to the application of Sobolev embedding to deduce uniform upper and lower bounds for the density.
This BD entropy approach has been extended to the isothermal NS system \cite{EEKO-ISO}, and further to the Navier--Stokes--Korteweg system \cite{EKK-NSK}.
Despite its strength and versatility, the BD framework has seen comparatively limited direct application to the full NSF system.
In the context of the BNSF system, although the constitutive relation \eqref{umuv} resembles the BD structure, the above approach cannot be applied as follows.
While the BD entropy yields an \(L^\infty_t L^2_x\) bound for the density gradients in the barotropic setting, Brenner's model leads only to an \(L^2_t L^2_x\) bound for the specific volume gradients, which is substantially weaker and insufficient to use Sobolev embedding to get lower and upper bounds for the specific volume.

\vspace{2mm}
On the other hand, the NSF system, which exhibits a more intricate structure, has also been extensively studied by many authors.
In the pioneering work of Kazhikhov--Shelukhin \cite{KS-77}, the initial-boundary problem for the NSF system was investigated, where the existence of both weak and classical solutions was established for large initial data.
A corresponding result in the whole space \(\RR\), as considered in the present work, was subsequently obtained in \cite{Kazhikhov-82}.
As emphasized in \cite{AKM-90,Dafermos-SIMA82,TYZZ-SIMA13}, similarly to the aforementioned barotropic case, a key analytical difficulty lies in deriving uniform upper and lower bounds for the specific volume \(v\), as well as for the absolute temperature \(\th\).

\vspace{2mm}
Broadly speaking, two main approaches have been developed in the literature to address the existence theory for the NSF system.
The first approach, originating from \cite{KS-77}, is based on deriving a subtle representation formula for the specific volume as follows: for each \(n\in\mathbb{Z}\) and any \(x\in[n,n+1]\), 
\begin{equation} \label{rep}
v(t,x)=\frac{1+\frac{R}{\m}\int_0^t \th(s,x)B_n(s,x)Y_n(s)ds}{B_n(t,x)Y_n(t)}
\end{equation}
for some carefully chosen \(B_n(t,x)\) and \(Y_n(t)\).
This representation allows one to establish a positive lower bound for \(v\), from which, together with a standard maximum principle, a positive lower bound for \(\th\) follows.

\vspace{2mm}
The second approach, developed in \cite{TYZZ-SIMA13}, is motivated by the physical observation that the transport coefficients---namely, the viscosity and the heat conductivity---typically depend on the temperature.
In particular, the Chapman--Enskog expansion (see \cite{CC-Book90}) suggests that constitutive relations of the form 
\[
\m,\,\k\, \sim \,\th^\b
\]
for some \(\b\in[1/2,\infty)\) depending on the intermolecular potential.
Since the first approach is not well suited to handling degenerate transport coefficients (for instance, non-constant viscosity), considerable effort has been devoted to overcoming this limitation; see \cite{TYZZ-SIMA13,LYZZ-SIMA14}.
These works fall within the framework of Nishida--Smoller type results (see the original work \cite{NS-CPAM73} and subsequent advancements by Liu \cite{Liu-Indiana77} and Temple \cite{Temple-JDE81}), where the size of the admissible initial data depends on the smallness of \(\g-1\).
A key observation underlying this method is that when the viscosity depends on the specific volume, one can find that 
\[
\left(\frac{\m(v)v_x}{v}\right)_t
=\left(\frac{\m(v)v_t}{v}\right)_x
=\left(\frac{\m(v)u_x}{v}\right)_x
=u_t + p_x,
\]
and when the viscosity depends on the temperature, one can deduce 
\[
\left(\frac{\m(\th)v_x}{v}\right)_t
=u_t + p_x
+\frac{\m'(\th)}{v}(\th_t v_x-u_x\th_x).
\]
However, for the BNSF system, the continuity equation becomes nonlinear, and thus the effective interchangeability between \(v_t\) and \(u_x\) is no longer available.
As a result, despite the dissipative structure of the continuity equation, these observations do not directly apply in this setting.
The concise linear structure of the mass conservation law is also crucial in the derivation of the representation \eqref{rep}, and therefore the first approach is not applicable for the BNSF model as well.

\vspace{2mm}
To obtain uniform upper and lower bounds for the specific volume, it is essential to exploit the dissipative structure of the mass conservation law.
To this end, we apply the parabolic De Giorgi method developed in \cite{CV-Annals10} (see also \cite{Vasseur-LN}).
It was originally introduced for elliptic equations to obtain \(C^\a\) bounds \cite{DeGiorgi-1957} and was was extended to the parabolic setting in \cite{CV-Annals10}.
This enables us to upgrade an \(L^2_t L^2_x\) bound on \(v_x\) to \(L^\infty_t L^\infty_x\) bounds for both \(v\) and \(1/v\).
We also refer to \cite{CKV-JMPA20} for further applications of the parabolic De Giorgi method.

\vspace{2mm}
The remainder of the paper is organized as follows.
Section \ref{sec:2} presents two main propositions from which the main theorem, i.e., the global existence is derived.
In Section \ref{sec:EST}, we prove one of the main proposition, namely the \textit{a priori} estimates; in particular, we employ the De Giorgi method for parabolic equations in order to obtain upper and lower bounds for the \(v\) variable.
Appendix \ref{appendix:LWP} is devoted to the proof of the local well-posedness.

\section{Proof of Theorem \ref{thm:main}} \label{sec:2}
\setcounter{equation}{0}
In this section, we present two main propositions: one on the local-in-time existence and the other on \textit{a priori} uniform estimates, and show how they imply the main theorem.
Note that the former follows from a standard argument, while the latter relies on a combination of refined analytical techniques, including the De Giorgi method and delicate energy estimates, and constitutes the main contribution of this paper.

\subsection{Main Propositions}
To state the propositions, we introduce the following notation for the size of the initial data.
Let \(k\in\mathbb{N}\) be given and assume that the initial data \((v_0,u_0,\th_0)\) satisfies \((v_0-1,u_0,\th_0-1)\in (H^k(\RR))^3\).
Then, for each integer \(l\) with \(0\le l \le k\), we define
\[
\norm{v_0, u_0, \th_0}_l = \norm{v_0-1}_{H^l(\RR)} + \norm{u_0}_{H^l(\RR)} + \norm{\th_0-1}_{H^l(\RR)}.
\]
Moreover, \(M(\cdot,\ldots,\cdot)\) denotes a positive, finite function that is non-decreasing in each of its arguments, while \(C\) denotes a positive constant depending only on \(R,\g,\t,\m\) and \(\k\).

\vspace{2mm}
We first present the local-in-time existence.

\begin{prop}[Local Well-Posedness]\label{prop:LWP}
Let \(\t,\m,\k\) be any positive constants and \(k\ge 1\) be an integer.
For a given initial data \(U_0\coloneqq (v_0,u_0,\th_0)\), we assume that 
\[
\norm{U_0}_k \coloneqq \norm{v_0, u_0, \th_0}_k < \infty, \qquad
\ell_0 \coloneqq \min \{\inf_{x\in\RR} v_0(x), \inf_{x\in\RR}\th_0(x)\} > 0.
\]
Then, there exist \(T_0>0\) which is bounded below by a positive constant depending only on \(\norm{U_0}_k\) and \(\ell_0\), and a constant \(M_0 > 0\) independent of \(\norm{U_0}_k\) and \(\ell_0\) such that the following holds: the system \eqref{bnsfsys} has a unique strong solution \(U = (v,u,\th)\) on \([0,T_0]\) subject to the initial data \(U_0\) which satisfies the following: 
\[
(v-1,u,\th-1) \in \big(C([0,T_0];H^k(\RR)) \cap L^2(0,T_0;H^{k+1}(\RR))\big)^3,
\]
together with the size control 
\[
\norm{U}_k(t) \le M_0\norm{U_0}_k \qquad \forall t \in [0,T_0],
\]
\[
\int_0^{T_0}\int_\RR (\rd_x^{k+1}v)^2 + (\rd_x^{k+1}u)^2 + (\rd_x^{k+1}\th)^2 dx dt \le M_0\norm{U_0}_k^2,
\]
and the lower bounds on \(v\) and \(\th\) variables
\[
v(t,x), \th(t,x) \ge \frac{\ell_0}{2} \qquad \forall (t,x)\in [0,T_0] \times\RR,
\]
where \(\norm{U}_k(t)=\norm{v,u,\th}_k(t)\) are defined as follows:
\[
\norm{U}_k(t)
=\norm{v,u,\th}_k(t)
=\norm{v(t)-1}_{H^k(\RR)} + \norm{u(t)}_{H^k(\RR)} + \norm{\th(t)-1}_{H^k(\RR)}.
\]
\end{prop}

\begin{proof}
The proof is standard and is therefore relegated to Appendix \ref{appendix:LWP}.
\end{proof}

We now turn to the other proposition, which provides \textit{a priori} uniform estimates.

\begin{prop}[\textit{A Priori} Uniform Estimates]\label{prop:est}
Let \(\t,\m,\k\) be any positive constants and \(k\ge 3\) be an integer.
If \(U\coloneqq (v,u,\th)\) is a solution of \eqref{bnsfsys} on \([0,T)\) for some \(T>0\) subject to the initial data \(U_0\coloneqq(v_0,u_0,\th_0)\) such that 
\begin{align*}
(v-1, u, \th-1) &\in \big(C([0,t];H^k(\RR)) \cap L^2(0,t;H^{k+1}(\RR))\big)^3,\\
0<v^{-1}, \th^{-1} &\in L^\infty(0,t;L^\infty(\RR)), \qquad \forall t \in(0,T),
\end{align*}
then the following inequalities are satisfied:
\begin{align}
\sup_{t\in[0,T)} \norm{U}_k(t) &\le M(\norm{U_0}_k, \ell_0^{-1}, T), \label{est:Hk} \\
\norm{1/v}_{L^\infty([0,T)\times\RR)}, \norm{1/\th}_{L^\infty([0,T)\times\RR)} &\le M(\norm{U_0}_0, \ell_0^{-1}, T). \label{est:lbd}
\end{align}
\end{prop}

The first proposition is formulated for initial data in \(H^k\) with \(k\ge 1\), whereas the second requires initial data in \(H^k\) with \(k\ge 3\).
This higher regularity is needed in order to apply the maximum principle when deriving a lower bound for the temperature \(\th\).
As a consequence, we establish the existence of global classical solutions for initial data in \(H^3\).

\subsection{Proof of Theorem \ref{thm:main}}
The proof relies on a standard continuation argument.
Let \(k\ge3\) be an integer.
We first note that Proposition \ref{prop:LWP} ensures the local-in-time existence of a solution \((v,u,\th)\) to \eqref{bnsfsys} on \([0,T]\) for some \(T>0\), corresponding to the initial data \((v_0,u_0,\th_0)\), such that 
\[
(v-1,u,\th-1) \in \big(C([0,T];H^k(\RR)) \cap L^2(0,T;H^{k+1}(\RR))\big)^3
\]
and 
\[
v,\th > 0 \text{ on } [0,T], \quad \text{and} \quad v^{-1}, \th^{-1} \in L^\infty(0,T;L^\infty(\RR)).
\]
To establish global existence, let \(T^*\) be the supremum of all such \(T\) for which a solution with the aforementioned properties exists.
Then, Proposition \ref{prop:est} implies that the solution satisfies the estimate 
\begin{align*}
\sup_{t\in[0,T^*)} \norm{U}_k(t) &\le M(\norm{U_0}_k, \ell_0^{-1}, T^*), \\
\norm{1/v}_{L^\infty([0,T^*)\times\RR)}, \norm{1/\th}_{L^\infty([0,T^*)\times\RR)} &\le M(\norm{U_0}_0, \ell_0^{-1}, T^*).
\end{align*}
Assume \(T^* < \infty\).
With this bound, for any \(t < T^*\), the solution can be extended until time \(t + \d\), where \(\d\) is the local existence time provided by Proposition \ref{prop:LWP} by taking the solution at time \(t\) as the new initial data.
From the above uniform estimate, we know that this \(\d\) can be chosen to be independent of \(t\), hence we can extend the solution beyond \(T^*\), which contradicts the definition of \(T^*\).
This shows that \(T^* = \infty\), which means that the solution exists globally in time. 
The uniqueness follows from the same argument as in the proof of the local well-posedness in Appendix \ref{appendix:LWP}.
This completes the proof of Theorem \ref{thm:main}. \qed

\section{Proof of Proposition \ref{prop:est}: \textit{A Priori} Estimates}\label{sec:EST}
\setcounter{equation}{0}
\subsection{Relative Entropy and Its Balance Law} \label{subsec:Balance}
This section is devoted to proving Proposition \ref{prop:est}.
To this end, we first derive a uniform-in-time estimate from a balance law for the relative entropy.
The entropy of the system is given by 
\[
\et = -R\log v - \frac{R}{\g-1}\log \th.
\]
However, since the entropy is not sign-definite and may fail to be integrable for a general background state (other than \((1,0,1)\)), it is natural to work with the relative entropy instead.
The relative entropy with respect to the background state \(\Ubar\coloneqq (1,0,1)\) is defined as follows: (see \cite[Appendix A]{KVW-ARMA25})
\begin{equation}\label{relent}
\et(U|\Ubar)
=\et(U) - \et(\Ubar) -\nabla\et(\Ubar)(U-\Ubar)
= R\Phi(v) + \frac{R}{\g-1}\Phi(\th) + \frac{u^2}{2},
\end{equation}
where \(\Phi\) is a nonnegative convex function defined by 
\[
\Phi(z) = z-1-\log z, \qquad z>0.
\]
The relative entropy enjoys several useful properties, which we briefly summarize here.
First, due to the convexity of the entropy, the definition \eqref{relent} ensures positive definiteness.
Moreover, since the function \(\Phi\) exhibits a locally quadratic behavior, the relative entropy inherits a similar structure.
More precisey, for any \(0<a<b<\infty\), there exists a constant \(C=C(a,b)>0\) such that
\begin{equation*}
C^{-1}\abs{z-1}^2 \le \Phi(z) \le C \abs{z-1}^2, \qquad \forall z\in[a,b].
\end{equation*}
Hence, for an initial datum \(U_0=(v_0,u_0,\th_0)\) such that \(v_0\) and \(\th_0\) are bounded above and below away from zero, the following holds:
\begin{equation}\label{LQ}
C^{-1} \norm{U_0}_0 \le \int_\RR \et(U_0|\Ubar) dx \le C \norm{U_0}_0.
\end{equation}
Furthermore, it is easy to check that there exists a constant \(C>0\) such that
\begin{equation}\label{Qlin}
\begin{aligned}
&\Phi(z) \ge Cz, \qquad &&\forall z\ge2,\\
&\Phi(z) \ge C, \qquad &&\forall z\in(0,1/2).
\end{aligned}
\end{equation}

On the other hand, since the coefficients do not affect our analysis, we may assume without loss of generality that \(\t=\m=\k=1\).

\vspace{2mm}
In the following analysis, we repeatedly use integration by parts.
However, the vanishing of the boundary terms at the far fields \(\pm\infty\) cannot be directly justified due to the limited regularity.
To make the argument rigorous, we introduce a smooth cutoff function \(\phi_R\) for each \(R>0\), such that
\[
\phi_R(x)=
\begin{cases}
1 &\text{if } x\in[-R,R]\\
0 &\text{if } x\in \RR \setminus [-2R,2R]
\end{cases}
\]
and smoothly interpolates on the transition regions.
Moreover, it can be chosen such that
\[
\abs{\phi_R'(x)} \le 2/R, \qquad \forall x\in[-2R,-R] \cup [R,2R].
\]
We then perform the energy estimate for \(f \cdot \phi_R\) instead of \(f\), and pass to the limit as \(R\to\infty\) to obtain the desired result.
Thus, the boundary terms can be rigorously justified to vanish.

\vspace{2mm}
We now examine the balance law for the relative entropy and obtain the following lemma, which provides a uniform-in-time estimate.

\begin{lem}[Balance Law for the Relative Entropy]\label{lem:balancelaw}
Let \(U=(v,u,\th)\) be a solution of the system \eqref{bnsfsys} on \([0, T)\) subject to the initial data \(U_0=(v_0,u_0,\th_0)\) with \(v_0\) and \(\th_0\) positive everywhere,
\begin{equation}\label{Ecaldef}
\Ecal_0 \coloneqq \int_\RR \et(U_0|\Ubar) dx
= \int_\RR R\Phi(v_0(x)) + \frac{R}{\g-1}\Phi(\th_0(x)) + \frac{u_0^2(x)}{2} dx < \infty.
\end{equation}
If \(U=(v,u,\th)\) satisfies 
\[
(v-1,u,\th-1) \in \big(C([0,T];H^1(\RR)) \cap L^2(0,T;H^2(\RR))\big)^3,
\]
and \(v\) and \(\th\) do not vanish, the following inequality holds:
\begin{equation}\label{balancelaw}
\sup_{t\in [0, T]} \int_\RR \et(U|\Ubar)(t, x) dx
+ \int_0^T \int_\RR R\frac{(v_x)^2}{v^3} + \frac{(u_x)^2}{v\th} + \frac{(\th_x)^2}{v\th^2} dx dt
\le C\Ecal_0.
\end{equation}
In particular, the following also holds:
\begin{equation}\label{vthL1uL2}
\sup_{t\in [0, T)} \norm{v\one{v\ge 2}}_{L^1(\RR)} + \norm{\th\one{\th\ge 2}}_{L^1(\RR)}
+ \norm{u}_{L^2(\RR)}^2 \le C\Ecal_0,    
\end{equation}
and
\begin{equation}\label{weightedH1}
\int_0^T \int_\RR \frac{(v_x)^2}{v^3} + \frac{(u_x)^2}{v\th} + \frac{(\th_x)^2}{v\th^2} dx dt
\le C\Ecal_0.
\end{equation}
\end{lem}
\begin{proof}
We begin by deriving the time derivative of the relative entropy:
\begin{align*}
\rd_t \et(U|\Ubar)
&= R\Big(1-\frac{1}{v}\Big)v_t + \frac{R}{\g-1}\Big(1-\frac{1}{\th}\Big)\th_t + u u_t \\
&= R\Big(1-\frac{1}{v}\Big)u_x - \Big(1-\frac{1}{\th}\Big)pu_x - up_x \\
&\qquad
+ R\Big(1-\frac{1}{v}\Big)\Big(\frac{v_x}{v}\Big)_x 
+ \Big(1-\frac{1}{\th}\Big)\Big(\frac{\th_x}{v}\Big)_x 
+ \Big(1-\frac{1}{\th}\Big)\frac{(u_x)^2}{v}
+ u \Big(\frac{u_x}{v}\Big)_x\\
&= \Big(Ru - pu + \frac{uu_x}{v}\Big)_x 
+ R\Big(1-\frac{1}{v}\Big)\Big(\frac{v_x}{v}\Big)_x 
+ \Big(1-\frac{1}{\th}\Big)\Big(\frac{\th_x}{v}\Big)_x 
- \frac{(u_x)^2}{v\th}.
\end{align*}
Integrating over space, the flux term vanishes, and integration by parts yields
\begin{align*}
\frac{d}{dt} \int_\RR \et(U|\Ubar)(t, x) dx
&= \int_\RR R\Big(1-\frac{1}{v}\Big)\Big(\frac{v_x}{v}\Big)_x 
+ \Big(1-\frac{1}{\th}\Big)\Big(\frac{\th_x}{v}\Big)_x 
- \frac{(u_x)^2}{v\th} dx \\
&= -\int_\RR R\frac{(v_x)^2}{v^3} + \frac{(\th_x)^2}{v\th^2} + \frac{(u_x)^2}{v\th} dx.
\end{align*}
This dissipative structure implies that for any \(t\in[0,T)\),
\[
\int_\RR \et(U|\Ubar)(t,x)\,dx
+\int_0^t\int_\RR R\frac{(v_x)^2}{v^3} + \frac{(\th_x)^2}{v\th^2} + \frac{(u_x)^2}{v\th} \,dx\,ds
=\Ecal_0.
\]
Since \(t\in[0,T)\) is arbitrary, \eqref{balancelaw} follows.
This with \eqref{Qlin} yields \eqref{vthL1uL2} and \eqref{weightedH1}.
\end{proof}

\subsection{De Giorgi Estimates and its Consequences} \label{subsec:DeGiorgi}
The next step is to derive uniform bounds on \(v\) and \(1/v\).
In the present setting, using only the bounds obtained in Lemma \ref{lem:balancelaw}, we implement the De Giorgi iteration to derive uniform bounds on \(v\) and \(1/v\).

\vspace{2mm}
In general, obtaining \(L^\infty_t L^\infty_x\) bounds requires \(L^\infty_t H^1_x\) estimates, which are not available at this stage of the argument.
However, the De Giorgi method exploits the parabolic structure of the system together with the bounds from Lemma \ref{lem:balancelaw}---roughly speaking, \(L^\infty_t L^2_x\) and \(L^2_t H^1_x\) estimates---to yield the desired \(L^\infty_t L^\infty_x\) bounds.
Note that the order of the argument below is crucial, since the dissipation terms involve weights that are not a priori bounded.

\vspace{2mm}
We first introduce a numerical lemma needed for the De Giorgi iteration.

\begin{lem}\label{lem:DeGiorgiNum} \cite[Lemma 1]{V-NODEA07}
Let \(C>0\) and \(\b>1\) be any constants.
Then, there exists a constant \(C_0=C_0(C,\b)\) such that for every sequence \(\{W_k\}_{k=0}^\infty\) verifying \(0<W_0<C_0\) and for every \(k\ge 0\), 
\[
0 \le W_{k+1} \le C^k W_k^\b,
\]
the following holds: 
\[
\lim_{k\to\infty} W_k = 0.
\]
\end{lem}

We now apply the De Giorgi method to obtain bounds on \(v\) and \(1/v\).

\begin{lem}[Uniform Bounds on \(v\) and \(1/v\)] \label{lem:DeGiorgi}
Under the same assumptions as Lemma \ref{lem:balancelaw}, the following holds for \(t\in [0, T)\):
\begin{equation}\label{vublb}
\begin{aligned}
\sup_{t\in [0, T)} \norm{v}_{L^\infty(\RR)}
&\le M(\norm{v_0}_{L^\infty(\RR)}, \Ecal_0, T)\\
\sup_{t\in [0, T)} \norm{1/v}_{L^\infty(\RR)}
&\le M(\Ecal_0, \ell_0^{-1}, T).
\end{aligned}
\end{equation}
\end{lem}
\begin{proof}
In the proof, we first obtain an upper bound for \(v\) and then a strictly positive lower bound.
Each bound is obtained via the De Giorgi method, which consists of two main steps.
In the first step, we introduce an energy sequence \(\{E_k\}\) and show that it is well-defined.
In the second step, we establish a nonlinear recursive structure for \(\{E_k\}\) so that Lemma \ref{lem:DeGiorgiNum} can be applied.

\vspace{2mm}
\(\bullet\) Uniform Bounds on the Upper Bound for \(v\):

\step{1}
Let \(N \ge 2\) be any constant with \(N > \norm{v_0}_{L^\infty(\RR)}\) and we consider \(\vbar^N \coloneqq (v-N)_+\).
Then, since we have
\[
(v-N)_t - \Big(\frac{(v-N)_x}{v}\Big)_x = u_x,
\]
it follows that 
\begin{align*}
\frac{1}{2}\frac{d}{dt}\int_\RR (\vbar^N)^2 dx
+\int_\RR \frac{((\vbar^N)_x)^2}{v}dx
&=\int_\RR \vbar^N u_x dx
=-\int_\RR (\vbar^N)_x u dx\\
&\le \frac{1}{4}\int_\RR \frac{((\vbar^N)_x)^2}{v}dx
+\int_\RR v u^2 \one{v\ge N}dx.
\end{align*}
Then, using \eqref{vthL1uL2}, we find that
\begin{equation}\label{vstep1}
\frac{1}{2}\frac{d}{dt}\int_\RR (\vbar^N)^2 dx
+\frac{3}{4}\int_\RR \frac{((\vbar^N)_x)^2}{v}dx
\le \norm{v\one{v\ge N}}_{L^\infty(\RR)} M(\Ecal_0).
\end{equation}
On the other hand, for each \(t\in[0,T)\), there exists \(y\) (we omit the \(t\) dependence) such that 
\begin{equation}\label{refpt}
\vbar^N(t,y)=(v-N)_+(t,y)=0
\end{equation}
because if not, \(\Ecal_0\) cannot be finite due to \eqref{Qlin}.
Then, for each \(t\in[0,T)\), we observe
\begin{align*}
\vbar^N(t,x)
&=\vbar^N(t,y)
+\int_y^x (\vbar^N)_z(t,z) dz\\
&\le \sqrt{\int_\RR \frac{((\vbar^N)_x)^2}{v}dx} \sqrt{\int_\RR v\one{v\ge N}dx}
\le M(\Ecal_0)\sqrt{\int_\RR \frac{((\vbar^N)_x)^2}{v}dx}.
\end{align*}
This together with \eqref{vstep1} implies that 
\[
\frac{1}{2}\frac{d}{dt}\int_\RR (\vbar^N)^2 dx
+\frac{3}{4}\int_\RR \frac{((\vbar^N)_x)^2}{v}dx
\le \bigg(N+M(\Ecal_0)\sqrt{\int_\RR \frac{((\vbar^N)_x)^2}{v}dx}\bigg) M(\Ecal_0).
\]
Thus, we obtain 
\[
\frac{1}{2}\frac{d}{dt}\int_\RR (\vbar^N)^2 dx
+\frac{1}{2}\int_\RR \frac{((\vbar^N)_x)^2}{v}dx
\le M(N,\Ecal_0).
\]
Then, by the integration with respect to \(t\) on both sides, we get 
\[
\sup_{t\in[0,T)} \int_\RR (\vbar^N)^2 dx
+\int_0^T \int_\RR \frac{((\vbar^N)_x)^2}{v} dx
\le M(N,\Ecal_0,T)
\]

We fix \(N > \norm{v_0}_{L^\infty(\RR)}\), and set \(M^*\) be the RHS of the above inequality with this fixed \(N\).
Let \(L\) be any constant with \(L > 2N\) (which will be chosen later) and consider a sequence 
\[
c_k \coloneqq L(1-2^{-k-1}), \qquad \forall k\in \mathbb{N}\cup\{0\}.
\]
Note that \(L > c_{k+1} > c_k \ge c_0 = L/2 > N\) for all \(k\ge0\) and that \(\lim_{k\to\infty} c_k=L\).

\vspace{2mm}
For each \(k\ge0\), we define 
\[
E_k \coloneqq
\sup_{t\in[0,T)} \int_\RR (\vbar^{c_k})^2 dx
+\int_0^T \int_\RR \frac{((\vbar^{c_k})_x)^2}{v} dx dt.
\]
Note that since \(c_0 = L/2 > N\), we have \(E_0 \le M^*\).
This shows that \(E_k\) is well-defined.

\step{2}
Using the observation in \eqref{vstep1}, we find that 
\[
\frac{1}{2}\frac{d}{dt}\int_\RR (\vbar^{c_k})^2 dx
+\frac{1}{2}\int_\RR \frac{((\vbar^{c_k})_x)^2}{v}dx
\le \norm{v\one{v\ge c_k}}_{L^\infty(\RR)} M(\Ecal_0).
\]
Integration over \(t\in[0,T)\) yields that 
\[
E_k \le M(\Ecal_0) \int_0^T \norm{v\one{v\ge c_k}}_{L^\infty(\RR)} dt.
\]
Then, we perform a non-linearization as follows.
We first note that for any positive \(a<b\), 
\begin{equation}\label{NL}
v \one{v\ge b} \le \frac{b}{(b-a)^2} (\vbar^a)^2.
\end{equation}
This implies that for any \(c'\) with \(c_{k-1}<c'<c_k\), we have 
\[
\norm{v\one{v\ge c_k}}_{L^\infty(\RR)}
\le \frac{c_k}{(c_k-c')^2} \big\|\vbar^{c'}\big\|_{L^\infty(\RR)}^2,
\]
whose right-hand side can be bounded by the following pointwise estimate: for any \(t\in[0,T)\), 
\begin{equation} \label{pwest}
(\vbar^{c'})^2 (t,x)
=\Big(\vbar^{c'}(t,y)+\int_y^x \rd_z(\vbar^{c'}(t,z)) dz\Big)^2
\le \int_\RR v\one{v\ge c'}dx \int_\RR \frac{((\vbar^{c'})_x)^2}{v}dx.
\end{equation}
Here, \(y\) is chosen in the same manner as \eqref{refpt}
This together with \eqref{NL} shows that 
\[
(\vbar^{c'})^2 (t,x)
\le \frac{c'}{(c'-c_{k-1})^2} \int_\RR (\vbar^{c_{k-1}})^2 dx
\int_\RR \frac{((\vbar^{c_{k-1}})_x)^2}{v}dx.
\]
Setting \(c'=\frac{c_k+c_{k-1}}{2}\), we obtain
\[
\norm{v\one{v\ge c_k}}_{L^\infty(\RR)}
\le \frac{16c_k^2}{(c_k-c_{k-1})^4} \int_\RR (\vbar^{c_{k-1}})^2 dx
\int_\RR \frac{((\vbar^{c_{k-1}})_x)^2}{v} dx.
\]
Hence, we find that 
\[
E_k
\le \frac{M_1^* c_k^2}{(c_k-c_{k-1})^4}
\sup_{t\in[0,T)}\int_\RR (\vbar^{c_{k-1}})^2 dx
\int_0^T \int_\RR \frac{((\vbar^{c_{k-1}})_x)^2}{v} dx dt
\le \frac{M_1^* L^2}{2^{-4(k+1)} L^4} E_{k-1}^2
\]
for some constant \(M_1^*=M(\Ecal_0)\).
We set \(M_2^* \coloneqq 16(16M_1^*+1)\) and \(F_k\coloneqq E_k/L^2\) for all \(k\ge0\).\\
Then, the above inequality can be rewritten as follows: 
\begin{equation}\label{NL-v}
F_k \le (M_2^*)^k F_{k-1}^2, \qquad \forall k\ge 0.
\end{equation}
Since \(E_0\le M^*\), it follows that 
\[
F_0 = \frac{E_0}{L^2} \le \frac{M^*}{L^2} \to 0 \qquad \text{as } L \to \infty.
\]
Thus, for \(C_0 = C_0(M_2^*,2)\) in Lemma \ref{lem:DeGiorgiNum}, we choose a large enough \(L\) such that \(F_0 < C_0\).
Therefore, applying Lemma \ref{lem:DeGiorgiNum} to \eqref{NL-v}, we conclude that \(\lim_{k\to\infty} F_k =0\) and 
\[
\lim_{k\to\infty} E_k = 0.
\]
This completes the De Giorgi estimate, which gives an \(L^\infty\) bound for \(v\) as follows: 
\[
\norm{v}_{L^\infty([0,T)\times\RR)} \le L.
\]
The upper bound \(L\) depends on \(M^*\), which in turn depends on \(\norm{v_0}_{L^\infty(\RR)},\Ecal_0\) and \(T\).

\vspace{2mm}
\(\bullet\) Uniform Bounds on the Lower Bound for \(v\):

\step{1}
First of all, we note that the continuity equation is equivalent to 
\[
2\Big(\frac{1}{\sqrt{v}}\Big)_t
=-\frac{v_t}{v\sqrt{v}}
=-\frac{u_x}{v\sqrt{v}} -\frac{1}{v\sqrt{v}}\Big(\frac{v_x}{v}\Big)_x.
\]
Let \(N\ge \sqrt{2}\) be any constant with \(N>\ell_0^{-1/2}\ge\norm{1/\sqrt{v_0}}_{L^\infty(\RR)}\).
Let \(w\coloneqq 1/\sqrt{v}\) and we introduce the following notation
\[
\wbar^N \coloneqq (w-N)_+ = \Big(\frac{1}{\sqrt{v}}-N\Big)_+.
\]
Then we observe
\begin{align*}
\frac{d}{dt}\int_\RR (\wbar^N)^2 dx
+\frac{1}{2}\int_\RR \frac{(v_x)^2}{v^4}\one{v\le1/N^2}dx
&\le -\frac{3}{2}\int_\RR \frac{u v_x}{v^2\sqrt{v}} \wbar^N dx
-\frac{1}{2}\int_\RR \frac{u v_x}{v^3}\one{v\le1/N^2}dx\\
&\le 2\int_\RR \frac{1}{v^3}\abs{u v_x}\one{v\le1/N^2}dx\\
&\le \frac{1}{4}\int_\RR \frac{(v_x)^2}{v^4}\one{v\le1/N^2}dx
+4\int_\RR \frac{u^2}{v^2}\one{v\le1/N^2}dx.
\end{align*}
Thanks to \eqref{vthL1uL2}, we obtain
\begin{equation}\label{vistep1}
\frac{d}{dt}\int_\RR (\wbar^N)^2 dx
+\frac{1}{4}\int_\RR \frac{(v_x)^2}{v^4}\one{v\le1/N^2}dx
\le \norm{\frac{1}{v}\one{v\le1/N^2}}_{L^\infty(\RR)}^2  M(\Ecal_0).
\end{equation}
Similarly to \eqref{refpt}, for each \(t\in[0,T)\), we choose \(y\) such that 
\[
\wbar^N (t,y)=0.
\]
Here, we used \eqref{Qlin} and \eqref{balancelaw}.
Then, we find that for any \(t\in[0,T)\),
\begin{align*}
\Big(\frac{1}{v}-N^2\Big)_+(t,x)
&=\Big(\frac{1}{v}-N^2\Big)_+(t,y)
+\int_y^x \frac{v_x}{v^2} \one{v\le 1/N^2} dz\\
&\le \sqrt{\int_\RR \frac{(v_x)^2}{v^4}\one{v\le1/N^2} dx}
\sqrt{\int_\RR \one{v\le1/N^2}dx}.
\end{align*}
This implies that 
\[
\norm{\frac{1}{v}}_{L^\infty(\RR)}^2
\le 2N^4 + 2\int_\RR \one{v\le1/N^2}dx \int_\RR \frac{(v_x)^2}{v^4} \one{v\le1/N^2}dx.
\]
Then, using \eqref{vistep1} and \eqref{balancelaw}, we obtain
\begin{align*}
&\frac{d}{dt}\int_\RR (\wbar^N)^2 dx
+\frac{1}{4}\int_\RR \frac{(v_x)^2}{v^4}\one{v\le1/N^2}dx\\
&\qquad\qquad\le M(N,\Ecal_0)
+M(\Ecal_0)\int_\RR \one{v\le1/N^2}dx \int_\RR \frac{(v_x)^2}{v^4} \one{v\le1/N^2}dx\\
&\qquad\qquad\le M(N,\Ecal_0) + \frac{M(\Ecal_0)}{\Phi(1/N^2)}\int_\RR \Phi(v)dx \int_\RR \frac{(v_x)^2}{v^4} \one{v\le1/N^2}dx.
\end{align*}
We apply \eqref{balancelaw} once again to find that 
\[
\frac{d}{dt}\int_\RR (\wbar^N)^2 dx
+\frac{1}{4}\int_\RR \frac{(v_x)^2}{v^4}\one{v\le1/N^2}dx
\le M(N,\Ecal_0) + \frac{M(\Ecal_0)}{\Phi(1/N^2)} \int_\RR \frac{(v_x)^2}{v^4} \one{v\le1/N^2}dx.
\]
Since the function \(\Phi\) is decreasing on \((0,1)\) and \(\Phi(1/N^2) \to \infty\) as \(N\to\infty\), we choose \(N_0\) large enough such that for any \(N\ge N_0\), the following holds:
\[
\frac{d}{dt}\int_\RR (\wbar^N)^2 dx
+\frac{1}{8}\int_\RR \frac{(v_x)^2}{v^4}\one{v\le1/N^2}dx
\le M(N,\Ecal_0).
\]
We take integration with respect to \(t\) and conclude that when \(N > \max\{N_0,\ell_0^{-1/2},\sqrt{2}\}\), 
\[
\sup_{t\in[0,T)} \int_\RR (\wbar^N)^2dx
+\int_0^T \int_\RR \frac{(v_x)^2}{v^4} \one{v\le1/N^2} dx dt
\le M(N,\Ecal_0,T).
\]

Again, we fix \(N > \max\{N_0,\ell_0^{-1/2},\sqrt{2}\}\), and set \(M^*\) to be the RHS above with this \(N\).
Let \(L\) be any constant with \(L>2N\) and we define 
\[
c_k \coloneqq L(1-2^{-k-1}), \qquad \forall k\in \mathbb{N}\cup\{0\},
\]
which satisfies \(L > c_{k+1} > c_k \ge c_0 = L/2 > N\) for every \(k\ge0\) and \(\lim_{k\to\infty} c_k=L\).

\vspace{2mm}
Finally, for each \(k\ge0\), we define the following energy function 
\[
E_k \coloneqq
\sup_{t\in[0,T)} \int_\RR (\wbar^{c_k})^2 dx
+\int_0^T \int_\RR \frac{(v_x)^2}{v^4} \one{v\le1/c_k^2} dx dt.
\]
Then, from \(c_0=L/2>N\), it follows that \(E_0\le M^*\).
Therefore, \(E_k\) is well defined.

\step{2} 
Likewise, we establish an iteration scheme.
To this end, using \eqref{vistep1}, we get
\[
\frac{d}{dt}\int_\RR (\wbar^{c_k})^2 dx
+\frac{1}{4}\int_\RR \frac{(v_x)^2}{v^4}\one{v\le1/c_k^2}dx
\le \norm{\frac{1}{v}\one{v\le1/c_k^2}}_{L^\infty(\RR)}^2  M(\Ecal_0).
\]
Then, the following nonlinearization inequality 
\[
\frac{1}{\sqrt{v}} \one{v\le1/c_k^2}
\le \frac{c_k}{(c_k-c_{k-1})^{3/2}} (\wbar^{c_{k-1}})^{3/2}
\]
yields that 
\[
\frac{d}{dt}\int_\RR (\wbar^{c_k})^2 dx
+\frac{1}{4}\int_\RR \frac{(v_x)^2}{v^4}\one{v\le1/c_k^2}dx
\le \frac{M(\Ecal_0) c_k^4}{(c_k-c_{k-1})^6} \big\|\wbar^{c_{k-1}}\big\|_{L^\infty(\RR)}^6.
\]
We take integration in \(t\) to obtain
\[
E_k \le \frac{M(\Ecal_0) c_k^4}{(c_k-c_{k-1})^6} \int_0^T \big\|\wbar^{c_{k-1}}\big\|_{L^\infty(\RR)}^6 dt.
\]
To establish the desired nonlinear recursive structure, we need the following pointwise estimate: using the same argument in \eqref{pwest}, we find that
\[
\norm{\wbar^c}_{L^\infty(\RR)}^3
\le \int_\RR 3(\wbar^c)^2 \Big|-\frac{v_x}{2v\sqrt{v}}\Big| \one{v\le 1/c^2}dx
\le C\sqrt{\int_\RR \frac{(v_x)^2}{v^4}\one{v\le 1/c^2}dx} \sqrt{\int_\RR (\wbar^c)^2 dx}.
\]
Hence, we obtain 
\begin{align*}
E_k &\le \frac{M_1^* c_k^4}{(c_k-c_{k-1})^6}
\sup_{t\in[0,T)}\int_\RR (\wbar^{c_{k-1}})^2 dx
\int_0^T \int_\RR \frac{(v_x)^2}{v^4}\one{v\le 1/c_{k-1}^2}dx dt \\
&\le \frac{M_1^* c_k^4}{(c_k-c_{k-1})^6} E_{k-1}^2
\le \frac{M_1^* L^4}{2^{-6(k+1)}L^6} E_{k-1}^2.
\end{align*}
for some \(M_1^*=M(\Ecal_0)\).
Let \(F_k \coloneqq E_k/L^2\) for each \(k\ge0\) and let \(M_2^*\coloneqq 64(64M_1^*+1)\).\\
Then, it follows that 
\[
F_k \le (M_2^*)^k F_{k-1}^2, \qquad \forall k\ge1.
\]
Following the same argument as for the upper bound, Lemma \ref{lem:DeGiorgiNum} yields the lower bound:
\[
\norm{1/v}_{L^\infty([0,T)\times\RR)} \le L.
\]
The lower bound \(L\) depends on \(M^*\), which itself depends on \(\ell_0^{-1},\Ecal_0\) and \(T\).
\end{proof}

The upper and strictly positive lower bounds for the specific volume \(v\) are now available.
With these bounds, we can obtain a pointwise bound for the temperature \(\th\) via the maximum principle, which in particular yields the lower bound for \(\th\) in \eqref{est:lbd}.

\begin{lem}[Uniform Bound on \(1/\th\)]\label{lem:Max}
Let \(U=(v,u,\th)\) be a solution of \eqref{bnsfsys} on \([0,T)\) for some \(T>0\) subject to the initial data \(U_0\coloneqq(v_0,u_0,\th_0)\) satisfying
\[
\norm{U_0}_0 \coloneqq \norm{v_0, u_0, \th_0}_0 < \infty, \qquad
\ell_0 \coloneqq \min \{\inf_{x\in\RR} v_0(x), \inf_{x\in\RR}\th_0(x)\} > 0.
\]
If the solution \(U=(v,u,\th)\) satisfies
\[
(v-1, u, \th-1) \in \big(C([0,t];H^3(\RR)) \cap L^2(0,t;H^4(\RR))\big)^3, \qquad \forall t\in(0,T),
\]
and \(v\) and \(\th\) do not vanish, the following uniform bound for \(t\in [0, T)\):
\begin{equation}\label{thlb}
\sup_{t\in [0, T)} \norm{1/\th}_{L^\infty(\RR)}
\le M(\Ecal_0, \ell_0^{-1}, T).
\end{equation}
\end{lem}

\begin{proof}
Let \(T_0\in(0,T)\) be any fixed constant.
To begin with, since
\[
\th \in C([0,T_0];H^3(\RR)) \cap L^2(0,T_0;H^4(\RR)), \qquad
\th_t \in C([0,T_0]; H^1(\RR)) \subset C([0,T_0]\times\RR),
\]
we have \(\th \in C^1([0,T_0]\times \RR)\).
Thus, we define a function \(\th_m\colon [0,T_0] \to \RR^+\) by 
\[
\th_m (t) = \inf_{x\in \RR} \th(t, x), 
\]
which is well defined and Lipschitz continuous; hence, by Rademacher's theorem, it is differentiable for almost every \(t\).
Let \(t\in [0, T_0]\) be the time that \(\th_m(t)<1\), then there exists \(x_t\in\RR\) such that \(\th_m(t)=\th(t,x_t)\).
We then claim that if \(\th_m'(t)\) exists, the following holds: 
\[
\th_m'(t) = \th_t(t, x_t). 
\]
This can be verified by following \cite{CDNP-20} for \(\mathbb{T}\) case; see also \cite{KV-JNS20} for \(\RR\) case.
If \(\th_m\) is differentiable at \(t\in(0,T)\), we have
\begin{align*}
\th_m'(t)
&= \lim_{h\to 0+} \frac{\th_m(t+h)-\th_m(t)}{h} \\
&= \lim_{h\to 0+} \frac{\th(t+h, x_{t+h})-\th(t, x_t)}{h} \\
&\le \lim_{h\to 0+} \frac{\th(t+h, x_t)-\th(t, x_t)}{h} = \th_t(t, x_t).
\end{align*}
On the other hand, the following holds:
\begin{align*}
\th_m'(t)
&= \lim_{h\to 0+} \frac{\th_m(t)-\th_m(t-h)}{h} \\
&= \lim_{h\to 0+} \frac{\th(t, x_t)-\th(t-h, x_{t-h})}{h} \\
&\ge \lim_{h\to 0+} \frac{\th(t, x_t)-\th(t-h, x_t)}{h} = \th_t(t, x_t).
\end{align*}
Thus, \(\th_m'(t)=\th_t(t,x_t)\) provided that \(\th_m'(t)\) exists.
Note that although \(x_{x\pm h}\) may be equal to \(\pm\infty\), the above inequalities still hold.
Hence, the energy (or the temperature) equation \eqref{bnsfsys} yields that for almost every \(t\),
\[
\th_m'(t)
= \th_t(t, x_t) 
= \frac{\g-1}{R}\Big(\Big(\frac{\th_x}{v}\Big)_x + \frac{(u_x)^2}{v} - \frac{R\th}{v}u_x\Big)(t, x_t).
\]
Motivated by \cite{KS-77}, using a perfect square form, we find that 
\[
\th_m'(t) \ge -\frac{\g-1}{4}\frac{R\th^2}{v}(t, x_t) + \frac{\g-1}{R} \Big(\frac{\th_x}{v}\Big)_x (t, x_t).
\]
Moreover, since \(\th_{xx}(t, x_t) \ge 0\) and \(\th_x(t, x_t)=0\), we apply \eqref{vublb} to obtain 
\[
\th_m'(t)
\ge -\frac{\g-1}{4}\frac{R\th^2}{v}(t, x_t)
\ge -M(\Ecal_0, \ell_0^{-1}, T) \th_m^2(t).
\]
This implies that 
\[
\Big(\frac{1}{\th_m(t)}\Big)' \le M(\Ecal_0, \ell_0^{-1}, T).
\]
Therefore, including the case when \(\th_m (t)\ge 1\), we have 
\[
\Big(\max\big(\frac{1}{\th_m(t)}, 1\big)\Big)' \le M(\Ecal_0, \ell_0^{-1}, T).
\]
Integrating over time, we obtain the desired bound on \(1/\th\): for any \(t\in[0,T)\),
\[
\th_m(t) \ge \frac{1}{M(\Ecal_0, \ell_0^{-1}, T)T + \max\big(\frac{1}{\th_m(0)},1\big)} \ge \frac{1}{M(\Ecal_0, \ell_0^{-1}, T)}.
\]
This completes the proof.
\end{proof}

\subsection{Higher Integrability and its Consequences} \label{subsec:ControlE}
Thanks to the results of the previous subsection, we can now remove some of the priorly imposed weights in the \(L^2_t H^1_x\) bounds in \eqref{weightedH1}.
However, this is still insufficient for deriving energy estimates for higher-order derivatives.
To address this, we require a certain level of higher integrability for the velocity and temperature, which allow us to effectively handle the nonlinear structure of the equations.
Note that, from this point on, we exploit time-integrated estimates, which play a crucial role in the remainder of the argument.

\begin{lem}\label{th1infty}
Under the same assumption as Lemma \ref{lem:balancelaw}, the following holds:
\[
\norm{\th}_{L^1(0, T; L^\infty(\RR))}
= \int_0^T \norm{\th(t)}_{L_x^\infty(\RR)} dt
\le M(\norm{v_0}_{L^\infty(\RR)}, \Ecal_0, T).
\]
\end{lem}

\begin{proof}
As in \eqref{pwest}, using \((\sqrt{\th})_x = \th_x / (2\sqrt{\th})\), we obtain
\begin{align*}
\|\sqrt{\th}\|_{L^\infty(\RR)}
&\le \sqrt{2} + \|\sqrt{\th} \one{\th\ge2}\|_{L^\infty(\RR)}
\le \sqrt{2} + \int_\RR \frac{\abs{\th_x}}{2\sqrt{\th}} \one{\th\ge2} dx\\
&\le \sqrt{2} + \sqrt{\int_\RR \frac{(\th_x)^2}{v\th^2}dx} \sqrt{\int_\RR v\th \one{\th\ge2} dx}.
\end{align*}
Then, using \eqref{vublb} and \eqref{vthL1uL2}, we find that 
\[
\|\sqrt{\th}\|_{L^\infty(\RR)}
\le \sqrt{2} + M(\norm{v_0}_{L^\infty(\RR)}, \Ecal_0, T) \sqrt{\int_\RR \frac{(\th_x)^2}{v\th^2}dx}.
\]
Squaring both sides and integrating over time, we get
\begin{align*}
\int_0^T \norm{\th}_{L^\infty(\RR)} dt
&\le M(\norm{v_0}_{L^\infty(\RR)}, \Ecal_0, T)
\Big(1 + \int_0^T\int_\RR \frac{(\th_x)^2}{v\th^2} dx dt \Big) \\
&\le M(\norm{v_0}_{L^\infty(\RR)}, \Ecal_0, T).
\end{align*}
Here, \eqref{weightedH1} was used.
This completes the proof.
\end{proof}

Then, we present the following lemma on higher integrabilities.
\begin{lem}[Higher Integrability of \(u\) and \(\th\)]\label{uthhigher}
Let \(U=(v,u,\th)\) be a solution of the system \eqref{bnsfsys} on \([0, T)\) subject to the initial data \((v_0,u_0,\th_0)\) with \(\norm{U_0}_1=\norm{v_0,u_0,\th_0}_1 < \infty\)
and 
\[
\ell_0 = \min \{\inf_{x\in\RR} v_0(x), \inf_{x\in\RR}\th_0(x)\} > 0.
\]
If \(U=(v,u,\th)\) satisfies 
\[
(v-1,u,\th-1) \in \big(C([0,T];H^1(\RR)) \cap L^2(0,T;H^2(\RR))\big)^3,
\]
and \(v\) and \(\th\) do not vanish, the following inequality holds for \(t\in [0,T)\):
\begin{equation}\label{uthL4L2}
\sup_{t\in [0, T)} \big(\norm{u}_{L^4(\RR)} + \norm{\th-1}_{L^2(\RR)}\big)
\le M(\norm{U_0}_1, \ell_0^{-1}, T),
\end{equation}
and the following integrability in time holds:
\begin{equation}\label{uthL2H1}
\int_0^T \norm{uu_x}_{L^2(\RR)}^2 + \norm{\th_x}_{L^2(\RR)}^2 dt
\le M(\norm{U_0}_1, \ell_0^{-1}, T).
\end{equation}
\end{lem}
\begin{proof}
First of all, we recall that the total energy \(E\) is given by \(E=\frac{R}{\g-1}\th+\frac{u^2}{2}\), and satisfies the following equation (see \eqref{bnsf0}): 
\[
E_t = \Big(-pu + \frac{\th_x}{v} + \frac{uu_x}{v}\Big)_x.
\]
Notice that due to the far-field condition on the initial data, in particular \(\th_0 \to 1\), it is natural to consider the quantity \(E-\frac{R}{\g-1}=\frac{R}{\g-1}(\th-1)+\frac{u^2}{2}\).
Thus, we proceed as follows:
\begin{align*}
\frac{d}{dt}\frac{1}{2}\int_\RR \Big(E-\frac{R}{\g-1}\Big)^2 dx
&=-\int_\RR E_x \Big(-pu + \frac{\th_x}{v} + \frac{uu_x}{v}\Big) dx\\
&\le -\frac{\g-1}{R} \int_\RR \frac{(E_x)^2}{v} dx
+C\int_\RR \abs{E_x} \abs{pu} + \abs{E_x} \frac{\abs{uu_x}}{v} dx\\
&\le -\frac{\g-1}{2R} \int_\RR \frac{(E_x)^2}{v} dx
+C\int_\RR \frac{u^2 \th^2}{v} dx + C\int_\RR \frac{(uu_x)^2}{v} dx.
\end{align*}
To offset the \((uu_x)^2\) term, we observe that
\begin{align*}
\frac{d}{dt}\frac{1}{4}\int_\RR u^4 dx
&=-\int_\RR u^3 p_x dx -3\int_\RR \frac{(uu_x)^2}{v} dx
=3\int_\RR u^2 u_x p dx -3\int_\RR \frac{(uu_x)^2}{v} dx\\
&\le -\int_\RR \frac{(uu_x)^2}{v} dx + C\int_\RR \frac{u^2\th^2}{v}dx.
\end{align*}
Then, following the idea of \cite{KS-77}, we consider \(\frac{1}{2}\big(E-\frac{R}{\g-1}\big)^2 + \frac{M}{4}u^4\) with sufficiently large \(M\) so that the above \((uu_x)^2\) term can be canceled.
To be precise, we have 
\begin{align*}
\frac{d}{dt}\int_\RR \frac{1}{2}\Big(E-\frac{R}{\g-1}\Big)^2 + \frac{M}{4}u^4 dx
+\frac{\g-1}{2R} \int_\RR \frac{(E_x)^2}{v} dx
+M_1^* \int_\RR \frac{(uu_x)^2}{v} dx
\le C\int_\RR \frac{u^2 \th^2}{v} dx
\end{align*}
for some constant \(M_1^*>0\).
Note that while we assume \(\m=\k=1\) without loss of generality, the constant \(M\) depends on \(\m\) and \(\k\); nevertheless, the argument remains the same.
To use Gr\"onwall's inequality, we need to control the right-hand side as follows.
By \eqref{vublb}, we obtain 
\begin{equation}\begin{aligned}\label{u2th2intable}
\int_\RR \frac{u^2 \th^2}{v} dx
&\le M(\Ecal_0,\ell_0^{-1},T) \big(\norm{\th}_{L^\infty(\RR)}+1\big) \int_\RR u^4 + u^2 + (\th-1)^2 dx\\
&\le M(\Ecal_0,\ell_0^{-1},T) \big(\norm{\th}_{L^\infty(\RR)}+1\big) \a(t),
\end{aligned}\end{equation}
where the function \(\a\) is defined by
\[
\a(t) \coloneqq \int_\RR \frac{1}{2}\big(E-\frac{R}{\g-1}\big)^2 + \frac{M}{4}u^4 dx + 1.
\]
Then, summing up the above estimates, we find that 
\[
\a'(t)
+\frac{\g-1}{2R} \int_\RR \frac{(E_x)^2}{v} dx
+M_1^* \int_\RR \frac{(uu_x)^2}{v} dx
\le M(\Ecal_0,\ell_0^{-1},T) \big(\norm{\th}_{L^\infty(\RR)}+1\big) \, \a(t).
\]
Thanks to the bound on \(\norm{\th}_{L^1_t L^\infty_x}\) in Lemma \ref{th1infty}, we use Gr\"onwall's lemma with \eqref{LQ} to get
\[
\sup_{t\in [0, T)} \big(\norm{u}_{L^4(\RR)} + \norm{\th-1}_{L^2(\RR)}\big)
\le M(\norm{U_0}_1, \ell_0^{-1}, T),
\]
which proves \eqref{uthL4L2}.
Moreover, \eqref{uthL4L2} with Lemma \ref{lem:DeGiorgi} establishes \eqref{uthL2H1}.
\end{proof}

Thanks to the lower and upper bounds on the specific volume and the locally quadratic structure \eqref{LQ}, the uniform relative entropy estimate in Lemma \ref{lem:balancelaw} together with \eqref{uthL4L2} yields a uniform bound on \(\|U_0\|_0\).

\subsection{Uniform Estimates for First Order Derivatives} \label{subsec:H1}
We now aim to obtain uniform-in-time control of \(v_x\).
Note that all available estimates for the derivatives are time-integrated; therefore, we need to upgrade them to uniform-in-time bounds.
To this end, we first remove the dependence on \(\th\) in the control of \(u_x\), as stated in the following lemma.

\begin{lem}\label{lem:uL2H1}
Under the same assumption as Lemma \ref{uthhigher}, the following holds:
\begin{equation}\label{uL2H1}
\int_0^T\int_\RR (u_x)^2(t, x) dx dt \le M(\norm{U_0}_1, \ell_0^{-1}, T).
\end{equation}
\end{lem}
\begin{proof}
The momentum equation yields the following estimate: 
\begin{align*}
\frac{d}{dt}\int_\RR \frac{u^2}{2} dx
+\int_\RR \frac{(u_x)^2}{v}dx
&= -\int_\RR u p_x dx
=-R\int_\RR \frac{u}{v}\th_x dx
+R\int_\RR \frac{u\th}{v^2} v_x dx\\
&\le C\int_\RR (v_x)^2 + (\th_x)^2 dx
+C\int_\RR \frac{u^2}{v^2} dx
+C\int_\RR \frac{u^2 \th^2}{v^4} dx.
\end{align*}
Then, using \eqref{vublb}, we find that 
\[
\frac{d}{dt}\int_\RR \frac{u^2}{2} dx
+\int_\RR \frac{(u_x)^2}{v}dx
\le C\int_\RR (v_x)^2 + (\th_x)^2 dx
+ M(\Ecal_0,\ell_0^{-1},T) \int_\RR u^2 + u^2 \th^2 dx.
\]
Using \eqref{u2th2intable} together with \eqref{weightedH1} and \eqref{vublb}, the right-hand side is time-integrable.
Integrating in time, we obtain that:
\[
\sup_{t\in [0, T)} \int_\RR u^2 dx
+ \int_0^T\int_\RR \frac{(u_x)^2}{v} dx dt 
\le M(\norm{U_0}_1,\ell_0^{-1},T).
\]
Thanks to \eqref{vublb}, the desired bound can be obtained.
This completes the proof.
\end{proof}

We are now ready to prove the desired estimate for \(v_x\).

\begin{lem}\label{lem:vxLinfL2}
Under the same assumptions as Lemma \ref{uthhigher}, the following holds for \(t\in [0, T)\):
\[
\sup_{t\in [0, T)} \int_\RR \Big(\frac{v_x}{v}\Big)^2 dx
+ \int_0^T\int_\RR \frac{1}{v} \Big(\Big(\frac{v_x}{v}\Big)_x\Big)^2 dx dt
\le M(\norm{U_0}_1, \ell_0^{-1}, T).
\]
\end{lem}
\begin{proof}
Since we have 
\[
\Big(\frac{v_x}{v}\Big)_t = (\ln v)_{xt} = \Big(\frac{v_t}{v}\Big)_x,
\]
it follows that
\begin{align*}
\frac{d}{dt}\int_\RR \frac{1}{2}\Big(\frac{v_x}{v}\Big)^2 dx 
&= \int_\RR \frac{v_x}{v}\Big(\frac{v_t}{v}\Big)_x dx 
= -\int_\RR \Big(\frac{v_x}{v}\Big)_x \frac{u_x}{v} dx
- \int_\RR \frac{1}{v}\Big(\frac{v_x}{v}\Big)_x^2 dx \\
&\le -\frac{1}{2}\int_\RR \frac{1}{v} \Big(\frac{v_x}{v}\Big)_x^2  dx + \frac{1}{2}\int_\RR \frac{(u_x)^2}{v} dx.
\end{align*}
Thanks to Lemmas \ref{lem:DeGiorgi} and \ref{lem:uL2H1}, Gr\"onwall's lemma yields the desired conclusion.
\end{proof}

As a consequence, we obtain the following corollary.

\begin{cor} \label{cor:vd}
Under the same assumptions as Lemma \ref{uthhigher}, the following holds:
\begin{equation}\label{vxLinfL2}
\norm{v_x}_{L^\infty(0, T; L^2(\RR))} 
+ \norm{v_x}_{L^4(0, T; L^\infty(\RR))}
+ \norm{v_{xx}}_{L^2(0, T; L^2(\RR))}
\le M(\norm{U_0}_1, \ell_0^{-1}, T).
\end{equation}
\end{cor}
\begin{proof}
The first norm is controlled by Lemmas \ref{lem:DeGiorgi} and \ref{lem:vxLinfL2}.
To estimate the second one, since we have the upper bound for \(v\) in Lemma \ref{lem:DeGiorgi}, it is enough to control \(\norm{v_x/v}_{L^4(0, T; L^\infty(\RR))}\).
Since \(\norm{v_x/v}_{L^\infty(0, T; L^2(\RR))}\) and \(\norm{(v_x/v)_x}_{L^2(0, T; L^2(\RR))}\) are bounded, we get the desired bound on \(\norm{v_x/v}_{L^4(0, T; L^\infty(\RR))}\) by Sobolev embedding.
Lastly, using the bounds for the first and second norms, we deduce that \(\norm{v_x^2}_{L^2(0, T; L^2(\RR))}\) can be bounded. 
This, together with Lemma \ref{lem:vxLinfL2}, yields the control of \(\norm{v_{xx}}_{L^2(0, T; L^2(\RR))}\) as we desired.
\end{proof}

\subsection{Uniform Estimates for Higher Order Derivatives} \label{subsec:HigherOrder}
Now we are ready to establish the \(H^k\) bound of the solution, which yields \eqref{est:Hk}.
Exploiting the parabolic structures of the equations, we obtain control of higher-order derivatives by standard energy estimates, provided that the initial data is regular enough.
The proof proceeds by induction on the order of derivatives.
In what follows, we first establish the base case, which ensures a pointwise bound of \(\th\), i.e., a uniform-in-time bound of \(\norm{\th}_{L^\infty(\RR)}\).
Then, we carry out the induction step to derive bounds for higher-order derivatives.

\begin{lem}\label{lem:Hk}
Under the same assumption as Proposition \ref{prop:est}, the following holds:
\begin{equation}\label{thLinfty}
\sup_{t\in [0, T)} \norm{\th}_{L^\infty(\RR)} \le M(\norm{U_0}_1, \ell_0^{-1}, T),    
\end{equation}
and
\begin{equation}\label{solnHk}
\begin{aligned}
&\sum_{l\le k} \sup_{t\in [0, T)} \|\rd_x^l (v-1, u, \th-1)\|_{L^2(\RR)}
+ \sum_{l\le k} \|\rd_x^l (v-1, u, \th-1)\|_{L^4(0, T; L^\infty(\RR))} \\
&\qquad
+ \sum_{l\le k+1} \|\rd_x^l (v-1, u, \th-1)\|_{L^2(0, T; L^2(\RR))}
\le M(\norm{U_0}_k, \ell_0^{-1}, T).
\end{aligned}
\end{equation}
\end{lem}
\begin{proof}
Note that, by Sobolev embedding (as in the proof of Corollary \ref{cor:vd}), the second norm is bounded by the first and third norms.
So, it suffices to estimate the first and third norms.

\step{1}{(Base Case)}
We have already established the desired estimate for the case \(l=0\), as well as for \(v\) when \(l=1\).
To control \(u_x\) and \(\th_x\), we perform the following energy estimates:
\begin{align*}
\frac{d}{dt}\frac{1}{2} \int_\RR (u_x)^2 dx 
&= \int_\RR u_x u_{xt} dx 
=\int_\RR u_x \Big(-p_x + \Big(\frac{u_x}{v}\Big)_x\Big)_x dx \\
&= \int_\RR u_{xx} p_x dx - \int_\RR \frac{(u_{xx})^2}{v} dx + \int_\RR u_{xx}\frac{u_x v_x}{v^2} dx \\
&\le -c\int_\RR (u_{xx})^2 dx
+ M(\norm{U_0}_1,\ell_0^{-1},T) \int_\RR \th^2 (v_x)^2 + (\th_x)^2 + (u_x)^2 (v_x)^2 dx,
\end{align*}
for some constant \(c\) with \(1/c \le M(\norm{U_0}_1, \ell_0^{-1}, T)\).
To apply Gr\"onwall's inequality, we need to show that the right-hand side is time-integrable.
First, using \eqref{vxLinfL2} and \eqref{solnHk} with \(l = 0\), we find that
\begin{align*}
\int_0^T\int_\RR \th^2 (v_x)^2 dxdt
\le \int_0^T \Big(\norm{\th}_{L^\infty(\RR)}^2\int_\RR (v_x)^2 dx\Big) dt
\le \norm{\th}_{L^2(0, T; L^\infty(\RR))}^2 \norm{v_x}_{L^\infty(0, T; L^2(\RR))}^2
\end{align*}
is bounded.
Second, by \eqref{uthL2H1}, it simply follows that \(\int_\RR (\th_x)^2 dx\) is time-integrable.\\
Lastly, since we have 
\[
\int_\RR (u_x)^2 (v_x)^2 dx \le \norm{v_x}_{L^\infty(\RR)}^2 \int_\RR (u_x)^2 dx
\]
and \(\norm{v_x}_{L^\infty_x}^2\) is time-integrable by \eqref{vxLinfL2}, Gr\"onwall's lemma yields the desired bound on \(u_x\).

\vspace{2mm}
On the other hand, for \(\th_x\), we do a similar energy estimate as above, as follows:
\begin{align*}
&\frac{d}{dt} \frac{1}{2} \frac{R}{(\g-1)}\int_\RR (\th_x)^2 dx
= \int_\RR \th_x \frac{R\th_{xt}}{\g-1} dx
= \int_\RR \th_x \Big(-pu_x+\Big(\frac{\th_x}{v}\Big)_x+\frac{(u_x)^2}{v}\Big)_x dx\\
&\qquad\qquad
= - \int_\RR \frac{(\th_{xx})^2}{v} dx
+ \int_\RR p u_x \th_{xx} dx
+ \int_\RR \frac{v_x \th_x}{v^2} \th_{xx}dx
- \int_\RR \frac{(u_x)^2}{v} \th_{xx} dx \\
&\qquad\qquad
\le -c\int_\RR (\th_{xx})^2 dx
+ M(\norm{U_0}_1,\ell_0^{-1},T) \int_\RR \th^2 (u_x)^2 + (\th_x)^2 (v_x)^2 + (u_x)^4 dx,
\end{align*}
for some constant \(c\) with \(1/c \le M(\norm{U_0}_1, \ell_0^{-1}, T)\).
We, once again, analyze the right-hand side on a term-by-term basis.
First, \eqref{solnHk} with \(l=0\) and the above result on \(u_x\) yield that
\[
\int_0^T\int_\RR \th^2 (u_x)^2 dx dt
\le \norm{\th}_{L^2(0, T; L^\infty(\RR))}^2 \norm{u_x}_{L^\infty(0, T; L^2(\RR))}^2
\le M(\norm{U_0}_1,\ell_0^{-1},T).
\]
Second, the following holds: 
\[
\int_\RR (\th_x)^2 (v_x)^2 dx
\le \norm{v_x}_{L^\infty(\RR)}^2 \int_\RR (\th_x)^2 dx
\]
and \(\norm{v_x}_{L^\infty_x}^2\) is time-integrable by \eqref{vxLinfL2}.
Finally, using \eqref{solnHk} for \(u\) when \(l = 1\),
\begin{align*}
\int_0^T\int_\RR (u_x)^4 dx dt
&\le \int_0^T \Big(\norm{u_x}_{L^\infty(\RR)}^2 \int_\RR (u_x)^2 dx\big) dt
\le \norm{u_x}_{L^2(0, T; L^\infty(\RR))}^2 \norm{u_x}_{L^\infty(0, T; L^2(\RR))}^2
\end{align*}
is bounded.
Then, we apply Gr\"onwall's lemma to obtain the desired bound on \(\th_x\).

\vspace{2mm}
We have verified the estimates for \(k=1\), i.e., the base case.
In particular, this establishes \eqref{thLinfty}, which ensures that \(\th\) is uniformly bounded and thereby allows us to control the nonlinearities in the induction step for \(k>1\).

\step{2}{(Induction Step)}
We assume that the desired bounds hold for \(l\le k-1\), and aim to show that they also hold for \(k\).
We perform energy estimates by differentiating the system \(k\) times in space, multiplying the resulting equations by \(\rd_x^k (v, u, \th)\), and then integrating in space.
In what follows, \(\Mbar\) denotes a positive constant given by \(M(\norm{U_0}_k,\ell_0^{-1},T)\).

\vspace{2mm}
We first observe the \(v\) variable.
We proceed as follows:
\begin{align*}
\frac{d}{dt}\frac{1}{2}\int_\RR (\rd_x^k v)^2 dx
&= -\int_\RR (\rd_x^{k+1} v) (\rd_x^k u) dx
- \int_\RR (\rd_x^{k+1} v) \rd_x^k \Big(\frac{v_x}{v}\Big) dx \\
&\le -c\int_\RR (\rd_x^{k+1} v)^2 dx
+ \Mbar\int_\RR (\rd_x^k u)^2 + \abs{\Big[\rd_x^k, \frac{1}{v}\Big]v_x}^2 dx.
\end{align*}
The induction hypothesis shows that \(\int_\RR (\rd_x^k u)^2\) is time-integrable.
Moreover, for the commutator term, we apply \eqref{kpce} to find that
\[
\int_\RR \abs{\Big[\rd_x^k, \frac{1}{v}\Big]v_x}^2 dx
\le \Mbar \norm{v-1}_{H^k(\RR)}^2\norm{v_x}_{L^\infty(\RR)}^2.
\]
Then, since \(\norm{v_x}_{L^\infty_x}^2\) is time-integrable by the induction hypothesis, Gr\"onwall's lemma yields the desired bound on \(\rd_x^k v\).

\vspace{2mm}
Next, we analyze the \(u\) variable.
We observe that
\begin{align*}
\frac{d}{dt} \frac{1}{2}\int_\RR (\rd_x^k u)^2 dx
&= \int_\RR (\rd_x^{k+1} u) (\rd_x^k p) dx
- \int_\RR (\rd_x^{k+1} u) \rd_x^k \Big(\frac{u_x}{v}\Big) dx \\
&\le -c\int_\RR (\rd_x^{k+1} u)^2 dx
+ \Mbar\int_\RR (\rd_x^k p)^2 + \abs{\Big[\rd_x^k, \frac{1}{v}\Big]u_x}^2 dx.
\end{align*}
To handle \((\rd_x^k p)^2\) term, we apply the Sobolev product estimate to obtain
\begin{align*}
\Mbar \int_\RR (\rd_x^k p)^2 dx
&\le \Mbar\big(\norm{v-1}_{H^k(\RR)}^2 \norm{\th}_{L^\infty(\RR)}^2
+ \norm{v}_{L^\infty(\RR)}^2 \norm{\th-1}_{H^k(\RR)}^2\big)\\
&\le \Mbar \big(\norm{v-1}_{H^k(\RR)}^2 + \norm{\th-1}_{H^k(\RR)}^2\big),
\end{align*}
Here, \eqref{vublb} and \eqref{thLinfty} were used.
The right-hand side is time-integrable by the hypothesis.
In addition, for the commutator term, we have
\[
\int_\RR \abs{\Big[\rd_x^k, \frac{1}{v}\Big]u_x}^2 dx
\le \Mbar\big(\norm{v-1}_{H^k(\RR)}^2 \norm{u_x}_{L^\infty(\RR)}^2 
+ \norm{v_x}_{L^\infty(\RR)}^2\norm{u}_{H^k(\RR)}^2\big).
\]
The first term on the right-hand side is time-integrable from the above estimate on \(v\) and \eqref{solnHk} with \(l = 1\).
Since the second term is the multiplication of \(\norm{u}_{H^k_x}^2\) with the time-integrable term \(\norm{v_x}_{L^\infty_x}^2\), we use Gr\"onwall's lemma to get the desired bound on \(\rd_x^k u\).

\vspace{2mm}
It remains to estimate the \(\th\) variable.
To this end, we first note that
\begin{multline*}
\frac{d}{dt}\frac{1}{2}\frac{R}{\g-1} \int_\RR (\rd_x^k \th)^2 dx
= \int_\RR (\rd_x^{k+1}\th) \rd_x^{k-1}(pu_x) dx
- \int_\RR (\rd_x^{k+1} \th) \Big(\rd_x^k\Big(\frac{\th_x}{v}\Big) + \rd_x^k\Big(\frac{(u_x)^2}{v}\Big)\Big) dx \\
\le -c\int_\RR (\rd_x^{k+1} \th)^2 dx 
+ \Mbar\int_\RR \big(\rd_x^{k-1}(pu_x)\big)^2
+ \abs{\Big[\rd_x^k, \frac{1}{v}\Big]\th_x}^2
+ \abs{\rd_x^k\Big(\frac{(u_x)^2}{v}\Big)}^2 dx.
\end{multline*}
Using the Sobolev product estimate, the first term can be bounded as follows:
\[
\int_\RR \abs{\rd_x^{k-1}(pu_x)}^2 dx
\le \Mbar \big(\norm{v-1}_{H^{k-1}(\RR)}^2 \norm{u_x}_{L^\infty(\RR)}^2 
+ \norm{\th-1}_{H^{k-1}(\RR)}^2 \norm{u_x}_{L^\infty(\RR)}^2
+ \norm{u}_{H^k(\RR)}^2\big)
\]
and all of the terms on the right-hand side are time-integrable by the induction hypothesis.
The second term can be treated in exactly the same way as the commutator term in the estimate for \(u\).
For the last term, we apply the Sobolev product estimate to obtain
\[
\int_\RR \abs{\rd_x^{k-1}\Big(\frac{u_x^2}{v}\Big)}^2 dx
\le \norm{u_x}_{L^\infty(\RR)}^4 \|\rd_x^{k-1} v\|_{L^2(\RR)}^2 
+ \|\rd_x^k u\|_{L^2(\RR)}^2 \norm{u_x}_{L^\infty(\RR)}^2.
\]
The right-hand side is time-integrable by the induction hypothesis and the above estimate for \(u\).
Finally, we apply Gr\"onwall's inequality to obtain the desired bound on \(\rd_x^k \th\).
This completes the induction step, and thus the proof of Lemma \ref{lem:Hk}.
\end{proof}

\begin{appendix}
\section{Proof of Proposition \ref{prop:LWP}: Local Well-Posedness} \label{appendix:LWP}
\setcounter{equation}{0}
The proof of local well-posedness follows a standard approach (see for instance \cite{CEKS-26, CKV-JMPA20,CDNP-20}).
We first introduce an iteration scheme and establish uniform bounds for the resulting sequence.
Then we prove the convergence of the sequence, which yields the existence of solutions.
Finally, we show the uniqueness.
In the meantime, we also verify that the guaranteed existence time of the solution depends only on the \(H^k\) norm of the initial data and the lower bound of \(v_0\) and \(\th_0\), which is crucial for the proof of Theorem \ref{thm:main}.

\step{1}{(Iteration Scheme)}
First of all, we introduce an iteration scheme to construct a sequence of functions.
To this end, we recall that the initial data \((v_0,u_0,\th_0)\) satisfies 
\[
\norm{U_0}_k
=\norm{v_0,u_0,\th_0}_k
=\norm{v_0-1}_{H^k(\RR)}
+\norm{u_0}_{H^k(\RR)}
+\norm{\th_0-1}_{H^k(\RR)} < \infty
\]
for a fixed integer \(k\ge 1\), and assume that 
\[
\ell_0 \coloneqq \min \{\inf_{x\in\RR}v_0(x), \inf_{x\in\RR}\th_0(x)\} >0.
\]

We initialize the scheme as follows: 
\[
(v^0(t,x),u^0(t,x),\th^0(t,x))
\coloneqq (v_0(x),u_0(x),\th_0(x)).
\]
Then, for \(n\ge1\), and given \((v^{n-1},u^{n-1},\th^{n-1})\), we define \((v^n,u^n,\th^n)\) as a solution of the following system of linear equations:
\begin{equation} \label{lin-sys}
\begin{cases}
(v^n)_t - \t \big(\frac{(v^n)_x}{v^{n-1}}\big)_x
= (u^{n-1})_x, \\
(u^n)_t - \m \big(\frac{(u^n)_x}{v^{n-1}}\big)_x
= -R\big(\frac{\th^{n-1}}{v^{n-1}}\big)_x, \\
\frac{R}{\g-1}(\th^n)_t - \k \big(\frac{(\th^n)_x}{v^{n-1}}\big)_x
= -R\big(\frac{\th^{n-1}}{v^{n-1}}\big) (u^{n-1})_x
+\m\frac{((u^{n-1})_x)^2}{v^{n-1}},\\
(v^n,u^n,\th^n)\mid_{t=0} = (v_0,u_0,\th_0).
\end{cases}
\end{equation}
Thanks to the general theory of the heat equation, for given \((v^{n-1},u^{n-1},\th^{n-1})\) satisfying
\[
(v^{n-1}-1),u^{n-1},(\th^{n-1}-1)\in C(0,T;H^k(\RR))\cap L^2(0,T;H^{k+1}(\RR))
\]
for some \(T>0\), then \eqref{lin-sys} has a unique solution \((v^n,u^n,\th^n)\) belongs to the same space.
Note that \(k\ge 1\) is needed at this point, to have the \((u^{n-1}_x)^2\) term in the desired space.

We also set the following quantity to denote the energy of the initial data:
\[
\Etil_0 \coloneqq
\norm{v_0-1}_{H^k(\RR)}^2
+ \norm{u_0}_{H^k(\RR)}^2
+ \norm{\th_0-1}_{H^k(\RR)}^2.
\]

\step{2}{(Induction Hypotheses)}
We introduce a set of induction hypotheses on the sequence constructed above.
More precisely, we assume that there exist \(T_0>0\) and \(M_0\) (which will be chosen at the last step) such that, for some \(n\ge 1\), the iterate \((v^{n-1},u^{n-1},\th^{n-1})\) satisfies
\begin{equation}\label{indhyp-1}
\norm{v^{n-1}-1}_{C(0,T;H^k(\RR))}^2
+ \norm{u^{n-1}}_{C(0,T;H^k(\RR))}^2
+ \norm{\th^{n-1}-1}_{C(0,T;H^k(\RR))}^2
\le M_0\Etil_0,
\end{equation}
and
\begin{equation}\label{indhyp-2}
\int_0^{T_0} \int_\RR (\rd_x^{k+1} v^{n-1})^2 + (\rd_x^{k+1} u^{n-1})^2 + (\rd_x^{k+1} \th^{n-1})^2 dx dt \le M_0\Etil_0,
\end{equation}
together with
\begin{equation}\label{indhyp-3}
\inf_{[0,T_0]\times\RR}v^{n-1}(t,x) \ge \frac{\ell_0}{2}, \qquad
\inf_{[0,T_0]\times\RR}\th^{n-1}(t,x) \ge \frac{\ell_0}{2}.
\end{equation}
Note that this assumption is valid for the case when \(n = 1\).

\step{3}{(Uniform Bounds)}
We carry out the inductive step.
Under the induction hypotheses \eqref{indhyp-1}-\eqref{indhyp-3}, we show that they are preserved at the \(n\)-th iterate.

\vspace{2mm}
Let \(n\ge1\). For the sake of simplicity, we adopt the following notation: 
\[
(v,u,\th) \coloneqq (v^n,u^n,\th^n), \qquad
(\vtil,\util,\thtil) \coloneqq (v^{n-1},u^{n-1},\th^{n-1}).
\]
Then, the iteration scheme \eqref{lin-sys} can be written as follows:
\begin{equation} \label{iteration}
\begin{cases}
v_t - \t \big(\frac{v_x}{\vtil}\big)_x
= \util_x, \\
u_t - \m \big(\frac{u_x}{\vtil}\big)_x
= -R\big(\frac{\thtil}{\vtil}\big)_x, \\
\frac{R}{\g-1}\th_t - \k \big(\frac{\th_x}{\vtil}\big)_x
= -R\big(\frac{\thtil}{\vtil}\big) \util_x
+\m\frac{(\util_x)^2}{\vtil},\\
(v,u,\th)\mid_{t=0} = (v_0,u_0,\th_0).
\end{cases}
\end{equation}
Moreover, the induction hypotheses \eqref{indhyp-1}-\eqref{indhyp-3} read as follows:
\begin{equation}\label{ass1}
\norm{\vtil-1}_{C(0,T_0;H^k(\RR))}^2
+ \norm{\util}_{C(0,T_0;H^k(\RR))}^2
+ \|\thtil-1\|_{C(0,T_0;H^k(\RR))}^2
\le M_0\Etil_0,
\end{equation}
\begin{equation}\label{ass2}
\int_0^{T_0} \int_\RR (\rd_x^{k+1} \vtil)^2 + (\rd_x^{k+1} \util)^2 + (\rd_x^{k+1} \thtil)^2 dx dt \le M_0\Etil_0,
\end{equation}
and
\begin{equation}\label{ass3}
\inf_{[0,T_0]\times\RR} \vtil(t,x) \ge \frac{\ell_0}{2}, \qquad
\inf_{[0,T_0]\times\RR} \thtil(t,x) \ge \frac{\ell_0}{2}.
\end{equation}

We now derive uniform estimates for the inductive step.
To begin with, we observe
\begin{equation}\begin{aligned}\label{vl2-eq}
\rd_t \Big(\frac{1}{2}\int_\RR (v-1)^2 dx\Big)
&=\int_\RR (v-1)\Big(\t\big(\frac{v_x}{\vtil}\big)_x +\util_x\Big)dx\\
&=-\t\int_\RR \frac{(v_x)^2}{\vtil} dx 
-\int_\RR v_x \util dx
\le -\frac{\t}{2}\int_\RR \frac{(v_x)^2}{\vtil} dx
+ C\int_\RR \vtil \util^2 dx.
\end{aligned}\end{equation}
Then, since it follows from \eqref{ass1} that 
\[
\int_\RR \vtil \util^2 dx
\le \norm{\vtil}_{L^\infty(0,T_0;L^\infty(\RR))} \norm{\util}_{C(0,T_0;L^2(\RR))}^2
\le M(\Etil_0),
\]
we obtain
\[
\sup_{t\in[0,T_0]} \int_\RR (v-1)^2 dx
+ \t\int_0^{T_0} \int_\RR \frac{(v_x)^2}{\vtil} dx dt 
\le \int_\RR (v_0-1)^2 dx + M(\Etil_0) T_0
= \Etil_0 + M(\Etil_0) T_0.
\]
Thus, for \(T_0\) with \(M(\Etil_0)T_0 < \Etil_0\), we obtain 
\begin{equation}\label{vl2}
\sup_{t\in[0,T_0]} \int_\RR (v-1)^2 dx
\le 2\Etil_0.
\end{equation}

We now use \eqref{ass1} and observe
\begin{equation}\begin{aligned}\label{ul2-eq}
\rd_t \Big(\frac{1}{2}\int_\RR u^2 dx\Big)
&=-\m\int_\RR \frac{(u_x)^2}{\vtil} dx
+R\int_\RR u \Big(\frac{\thtil}{\vtil^2}\vtil_x-\frac{\thtil_x}{\vtil}\Big) dx\\
&\le -\m\int_\RR \frac{(u_x)^2}{\vtil} dx
+ \frac{1}{2}\int_\RR u^2 dx
+ M(\Etil_0,\ell_0^{-1}) \int_\RR (\vtil_x)^2 + (\thtil_x)^2 dx\\
&\le -\m\int_\RR \frac{(u_x)^2}{\vtil} dx
+ \frac{1}{2}\int_\RR u^2 dx
+ M(\Etil_0,\ell_0^{-1}).
\end{aligned}\end{equation}
Then, using Gr\"onwall's lemma, we find that for \(T_0>0\) of which smallness depends only on \(\Etil_0\) and \(\ell_0\),
\begin{multline}\label{ul2}
\sup_{t\in[0,T_0]} \int_\RR u^2 dx
\le e^{T_0}\int_\RR (u_0)^2 dx
+(e^{T_0}-1) M(\Etil_0,\ell_0^{-1}) \\
\le (1+2T_0)\Etil_0 + 2T_0 M(\Etil_0,\ell_0^{-1}) \le 2\Etil_0.
\end{multline}

We now turn to the \(\th\) variable.
The \(\th\)-equation in \eqref{iteration} gives that
\begin{multline*}
\rd_t \Big(\frac{R}{2(\g-1)}\int_\RR (\th-1)^2 dx\Big)
= -\k\int_\RR \frac{(\th_x)^2}{\vtil} dx
-R\int_\RR \frac{\thtil}{\vtil}\util_x (\th-1) dx
+\m\int_\RR \frac{(\util_x)^2}{\vtil} (\th-1) dx \\
= -\k\int_\RR \frac{(\th_x)^2}{\vtil} dx
+\int_\RR (\th-1)^2 dx
+M(\Etil_0,\ell_0^{-1}) \Big(1+\norm{\util_x}_{L^\infty(\RR)}^2\Big) \int_\RR (\util_x)^2 dx.
\end{multline*}
Then, using the following form of Sobolev inequality 
\begin{equation}\label{sob}
\norm{f}_{L^\infty(\RR)}^2 \le C \norm{f}_{L^2(\RR)} \norm{f'}_{L^2(\RR)}
\end{equation}
and \eqref{ass1}, we find that
\begin{equation}\begin{aligned}\label{thl2-eq}
\rd_t \Big(\frac{R}{2(\g-1)} \int_\RR (\th-1)^2 dx\Big)
&\le -\k\int_\RR \frac{(\th_x)^2}{\vtil} dx
+\int_\RR (\th-1)^2 dx\\
&\qquad\qquad
+M(\Etil_0,\ell_0^{-1})
+M(\Etil_0,\ell_0^{-1}) \norm{\util_{xx}}_{L^2(\RR)}.
\end{aligned}\end{equation}
Note by \eqref{ass2} that \(\norm{\util_{xx}}_{L^2(\RR)}^2\) is time-integrable.
Hence, Gr\"onwall's lemma establishes that for \(T_0>0\) of which smallness depends only on \(\Etil_0\) and \(\ell_0\),
\begin{equation}\begin{aligned}\label{thl2}
\sup_{t\in[0,T_0]} \int_\RR (\th-1)^2 dx
&\le e^{CT_0}\int_\RR (\th_0-1)^2 dx
+\frac{e^{CT_0}-1}{C} M(\Etil_0,\ell_0^{-1})\\
&\qquad\qquad
+ \sqrt{T_0} e^{CT_0} M(\Etil_0,\ell_0^{-1}) \norm{\util_{xx}}_{L^2(0,T_0; L^2(\RR))}
\le 2\Etil_0.
\end{aligned}\end{equation}

We now proceed to estimate the higher-order derivatives.
We first observe that
\begin{align*}
\rd_t \Big(\frac{1}{2}\int_\RR (v_x)^2 dx\Big)
&=-\t \int_\RR \frac{(v_{xx})^2}{\vtil} dx
+\t \int_\RR \frac{\vtil_x}{\vtil^2} v_x v_{xx} dx
-\int_\RR \util_x v_{xx} dx \\
&\le -\frac{\t}{2}\int_\RR \frac{(v_{xx})^2}{\vtil} dx
+M(\ell_0^{-1}) \int_\RR (v_x)^2 (\vtil_x)^2 dx
+C\int_\RR (\util_x)^2 dx.
\end{align*}
This together with \eqref{ass1} and \eqref{sob} allows one to use the diffusion term and to find that
\begin{align*}
\rd_t \Big(\frac{1}{2}\int_\RR (v_x)^2 dx\Big)
&\le -\frac{\t}{2}\int_\RR \frac{(v_{xx})^2}{\vtil} dx
+M(\ell_0^{-1}) \norm{v_x}_{L^2(\RR)} \norm{v_{xx}}_{L^2(\RR)} \int_\RR (\vtil_x)^2 dx
+M(\Etil_0)\\
&\le -\frac{\t}{4}\int_\RR \frac{(v_{xx})^2}{\vtil} dx
+M(\Etil_0,\ell_0^{-1}) \norm{v_x}_{L^2(\RR)}^2
+M(\Etil_0).
\end{align*}
By Gr\"onwall's lemma, we obtain
\begin{equation}\label{vh1}
\sup_{t\in[0,T]} \int_\RR (v_x)^2 dx
+ \int_0^T \int_\RR (v_{xx})^2 dx dt 
\le 2\Etil_0.
\end{equation}
Moreover, for the higher derivatives, we observe
\begin{align*}
\rd_t \Big(\frac{1}{2}\int_\RR (\rd_x^l v)^2 dx\Big)
&=-\t \int_\RR \frac{(\rd_x^{l+1} v)^2}{\vtil} dx
-\t \int_\RR (\rd_x^{l+1} v) \Big[\rd_x^l, \frac{1}{\vtil}\Big] (\rd_x v) dx\\
&\qquad\qquad
-\int_\RR (\rd_x^{l+1} v) (\rd_x^l \util) dx,
\end{align*}
where \(l\) is an integer which satisfies \(2\le l\le k\).
To proceed further, the following Kato--Ponce commutator estimate (see \cite{KP-CPAM88}) is needed:
\begin{equation}\label{kpce}
\norm{[\rd_x^k, f]g}_{L^2(\RR)}^2
\le C \norm{\rd_x f}_{L^\infty(\RR)}^2 \|\rd_x^{k-1}g\|_{L^2(\RR)}^2
+C \|\rd_x^k f\|_{L^2(\RR)}^2 \norm{g}_{L^\infty(\RR)}^2.
\end{equation}
Then, using \eqref{kpce} with \eqref{ass1} and \eqref{sob}, we obtain
\begin{align*}
\rd_t \Big(\frac{1}{2}\int_\RR (\rd_x^l v)^2 dx\Big)
&\le
-\frac{\t}{2} \int_\RR \frac{(\rd_x^{l+1} v)^2}{\vtil} dx
+C\norm{\vtil}_{L^\infty(\RR)}
\|\rd_x \frac{1}{\vtil}\|_{L^\infty(\RR)}^2 \|\rd_x^l v\|_{L^2(\RR)}^2\\
&\qquad
+C\norm{\vtil}_{L^\infty(\RR)}
\|\rd_x^l \frac{1}{\vtil}\|_{L^2(\RR)}^2 \|\rd_x v\|_{L^\infty(\RR)}^2
+C\norm{\vtil}_{L^\infty(\RR)} \int_\RR (\rd_x^l \util)^2 dx\\
&\le
-\frac{\t}{2} \int_\RR \frac{(\rd_x^{l+1} v)^2}{\vtil} dx
+M(\Etil_0,\ell_0^{-1}) \|\rd_x^l v\|_{L^2(\RR)}^2
+M(\Etil_0,\ell_0^{-1}).
\end{align*}
Note that when we control \(\norm{\rd_x v}_{L^\infty(\RR)}\), we use the Sobolev inequality \eqref{sob} when \(l=2\), while for \(l>2\), we rely on the estimate at \(l-1\) level.
We apply Gr\"onwall's lemma to obtain
\begin{equation}\label{vhm}
\sup_{t\in[0,T]} \int_\RR (\rd_x^l v)^2 dx
+ \int_0^T \int_\RR (\rd_x^{l+1} v)^2 dx dt 
\le 2\Etil_0.
\end{equation}

We now consider the \(u\) variable.
The second equation in \eqref{iteration} yields that
\begin{align*}
\rd_t \Big(\frac{1}{2}\int_\RR (u_x)^2 dx\Big)
&=-\m\int_\RR\frac{(u_{xx})^2}{\vtil} dx
+\m\int_\RR \frac{u_x u_{xx}}{\vtil^2} \vtil_x dx
+R\int_\RR u_{xx} \Big(\frac{\util}{\vtil}\Big)_x dx\\
&\hspace{-10mm}\le-\frac{\m}{2}\int_\RR\frac{(u_{xx})^2}{\vtil} dx
+M(\ell_0^{-1})\int_\RR (u_x)^2 (\vtil_x)^2 dx
+M(\Etil_0,\ell_0^{-1})\\
&\hspace{-10mm}\le-\frac{\m}{2}\int_\RR\frac{(u_{xx})^2}{\vtil} dx
+M(\ell_0^{-1}) \norm{u_x}_{L^2(\RR)} \norm{u_{xx}}_{L^2(\RR)} \int_\RR (\vtil_x)^2 dx
+M(\Etil_0,\ell_0^{-1})\\
&\hspace{-10mm}\le-\frac{\m}{4}\int_\RR\frac{(u_{xx})^2}{\vtil} dx
+M(\Etil_0,\ell_0^{-1}) \norm{u_x}_{L^2(\RR)}^2
+M(\Etil_0,\ell_0^{-1}).
\end{align*}
Here, we used \eqref{ass1} and \eqref{sob}.
This again with Gr\"onwall's lemma implies that
\begin{equation}\label{uh1}
\sup_{t\in[0,T_0]} \int_\RR (u_x)^2 dx
+ \int_0^{T_0} \int_\RR (u_{xx})^2 dx dt 
\le 2\Etil_0,
\end{equation}
for \(T_0>0\) of which smallness depends only on \(\Etil_0\) and \(\ell_0\).
In addition, for any integer \(l\) with \(2\le l\le k\), we observe
\begin{align*}
\rd_t \Big(\frac{1}{2}\int_\RR (\rd_x^l u)^2 dx\Big)
&=-\m\int_\RR\frac{(\rd_x^{l+1} u)^2}{\vtil} dx
-\m\int_\RR (\rd_x^{l+1} u) \Big[\rd_x^l, \frac{1}{\vtil}\Big] \rd_x u dx\\
&\qquad\qquad
+R\int_\RR (\rd_x^{l+1} u) \rd_x^l \Big(\frac{\util}{\vtil}\Big) dx.
\end{align*}
Then, using \eqref{ass1}, \eqref{sob} and \eqref{kpce}, we find that
\begin{align*}
\rd_t \Big(\frac{1}{2}\int_\RR (\rd_x^l u)^2 dx\Big)
&=-\frac{\m}{2}\int_\RR\frac{(\rd_x^{l+1} u)^2}{\vtil} dx
+C\norm{\vtil}_{L^\infty(\RR)}
\|\rd_x \frac{1}{\vtil}\|_{L^\infty(\RR)}^2 \|\rd_x^l u\|_{L^2(\RR)}^2\\
&\qquad\qquad
+C\norm{\vtil}_{L^\infty(\RR)}
\|\rd_x^l \frac{1}{\vtil}\|_{L^2(\RR)}^2 \|\rd_x u\|_{L^\infty(\RR)}^2
+M(\Etil_0,\ell_0^{-1})\\
&=-\frac{\m}{2}\int_\RR\frac{(\rd_x^{l+1} u)^2}{\vtil} dx
+M(\Etil_0,\ell_0^{-1}) \|\rd_x^l u\|_{L^2(\RR)}^2
+M(\Etil_0,\ell_0^{-1}).
\end{align*}
For \(\norm{\rd_x u}_{L^\infty(\RR)}\) we use the same strategy as in the \(v\) estimate. 
Then, similarly we obtain
\begin{equation}\label{uhm}
\sup_{t\in[0,T_0]} \int_\RR (\rd_x^l u)^2 dx
+ \int_0^T \int_\RR (\rd_x^{l+1} u)^2 dx dt 
\le 2\Etil_0,
\end{equation}
for \(T_0>0\) of which smallness depends only on \(\Etil_0\) and \(\ell_0\).

Finally, we address the \(\th\) variable.
To estimate the first-order derivative, we proceed as follows: using the third equation in \eqref{iteration} to note that
\begin{align*}
\rd_t \Big(\frac{R}{2(\g-1)}\int_\RR (\th_x)^2 dx\Big)
&=-\k\int_\RR \frac{(\th_{xx})^2}{\vtil} dx
+\k\int_\RR (\th_{xx}) (\th_x) \frac{\vtil_x}{\vtil^2} dx
+R\int_\RR (\th_{xx}) \frac{\thtil}{\vtil} \util_x dx\\
&\qquad\qquad
-\m\int_\RR (\th_{xx}) \frac{(\util_x)^2}{\vtil} dx\\
&\hspace{-40.8mm}\le -\frac{\k}{2}\int_\RR \frac{(\th_{xx})^2}{\vtil} dx
+M(\ell_0^{-1}) \int_\RR (\th_x)^2 (\vtil_x)^2 dx
+M(\Etil_0,\ell_0^{-1}) \int_\RR (\util_x)^2 dx
+M(\ell_0^{-1}) \int_\RR (\util_x)^4 dx.
\end{align*}
Then, using \eqref{ass1} and \eqref{sob}, we obtain
\begin{align*}
\rd_t \Big(\frac{R}{2(\g-1)}\int_\RR (\th_x)^2 dx\Big)
&\le -\frac{\k}{2}\int_\RR \frac{(\th_{xx})^2}{\vtil} dx
+M(\Etil_0,\ell_0^{-1}) \norm{\th_x}_{L^2(\RR)} \norm{\th_{xx}}_{L^2(\RR)}\\
&\qquad\qquad
+M(\Etil_0,\ell_0^{-1})
+M(\Etil_0,\ell_0^{-1})
\|\util_x\|_{L^2(\RR)} \|\util_{xx}\|_{L^2(\RR)}\\
&\hspace{-40.8mm}\le -\frac{\k}{4}\int_\RR \frac{(\th_{xx})^2}{\vtil} dx
+M(\Etil_0,\ell_0^{-1}) \norm{\th_x}_{L^2(\RR)}^2
+M(\Etil_0,\ell_0^{-1})
+M(\Etil_0,\ell_0^{-1}) \|\util_{xx}\|_{L^2(\RR)}.
\end{align*}
As in \eqref{thl2}, we exploit the time integrability of \(\|\util_{xx}\|_{L^2(\RR)}^2\) to obtain
\begin{equation}\label{thh1}
\sup_{t\in[0,T_0]} \int_\RR (\th_x)^2 dx
+ \int_0^{T_0} \int_\RR (\th_{xx})^2 dx dt 
\le 2\Etil_0,
\end{equation}
for \(T_0>0\) of which smallness depends only on \(\Etil_0\) and \(\ell_0\). 
Notice that when we control \(\norm{\util_{xx}}_{L^2(\RR)}\), we do not lose the lower bound on \(T_0\) because we have an \(L^2\)-in-time control on \(\norm{\util_{xx}}_{L^2(\RR)}\) and thus we have an extra \(\sqrt{T_0}\).

Further, the higher-order derivatives can be treated as follows: for \(l\in\mathbb{Z}\) with \(2\le l \le k\),
\begin{multline*}
\rd_t \Big(\frac{R}{2(\g-1)}\int_\RR (\rd_x^l \th)^2 dx\Big)
=-\k\int_\RR \frac{(\rd_x^{l+1} \th)^2}{\vtil} dx
-\k\int_\RR (\rd_x^{l+1} \th) \Big[\rd_x^l, \frac{1}{\vtil}\Big] \rd_x \th dx\\
+R\int_\RR (\rd_x^{l+1} \th) \rd_x^{l-1} \Big(\frac{\thtil}{\vtil}\util_x\Big)dx
-\m\int_\RR (\rd_x^{l+1} \th) \rd_x^{l-1} \Big(\frac{(\util_x)^2}{\vtil}\Big) dx.
\end{multline*}
It follows from Kato--Ponce commutator estimate \eqref{kpce} and the product estimate that
\begin{align*}
\rd_t \Big(\frac{R}{2(\g-1)}\int_\RR (\rd_x^l \th)^2 dx\Big)
&\le -\frac{\k}{2}\int_\RR \frac{(\rd_x^{l+1} \th)^2}{\vtil} dx
+C\norm{\vtil}_{L^\infty(\RR)} 
\|\rd_x \frac{1}{\vtil}\|_{L^\infty(\RR)}^2 \|\rd_x^l \th\|_{L^2(\RR)}^2\\
&\qquad\qquad
+C\norm{\vtil}_{L^\infty(\RR)} 
\|\rd_x^l \frac{1}{\vtil}\|_{L^2(\RR)}^2 \|\rd_x \th\|_{L^\infty(\RR)}^2
+M(\Etil_0,\ell_0^{-1})\\
&\le -\frac{\k}{2}\int_\RR \frac{(\rd_x^{l+1} \th)^2}{\vtil} dx
+M(\Etil_0,\ell_0^{-1}) \|\rd_x^l \th\|_{L^2(\RR)}^2
+M(\Etil_0,\ell_0^{-1}).
\end{align*}
The \(\util_x\) terms can be easily controlled by the bounds given in \eqref{ass1}, as we have \(L^\infty_x\) bound for the lower order derivatives and \(L^2_x\) bound for the highest order derivative.
This implies
\begin{equation}\label{thhm}
\sup_{t\in[0,T]} \int_\RR (\rd_x^l \th)^2 dx
+ \int_0^T \int_\RR (\rd_x^{l+1} \th)^2 dx dt 
\le 2\Etil_0.
\end{equation}
Therefore, combining \eqref{vl2}, \eqref{ul2}, \eqref{thl2}, \eqref{vh1}, \eqref{vhm}, \eqref{uh1}, \eqref{uhm}, \eqref{thh1}, \eqref{thhm}, we recover the first and second induction hypotheses \eqref{indhyp-1} and \eqref{indhyp-2} (or \eqref{ass1} and \eqref{ass2}).

\vspace{2mm}
It remains to verify the third induction hypothesis, namely, the lower bound for \(v\) and \(\th\).
To this end, we proceed as follows: since \(\norm{\rd_x f}_{H^{-1}(\RR)} \le \norm{f}_{L^2(\RR)}\), \eqref{iteration} yields
\[
\norm{v_t}_{H^{-1}(\RR)}
\le \norm{\util}_{L^2(\RR)}
+ M(\ell_0^{-1}) \norm{v_x}_{L^2(\RR)}
\le M(\Etil_0,\ell_0^{-1}).
\]
Then, we have 
\[
\norm{v(t,\cdot)-v_0(\cdot)}_{H^{-1}(\RR)}
\le \int_0^t \norm{v_t(s,\cdot)}_{H^{-1}(\RR)} ds
\le M(\Etil_0,\ell_0^{-1}) T_0.
\]
This together with the following basic functional inequality 
\[
\norm{f}_{L^2(\RR)}
\le C \norm{f}_{H^{-1}(\RR)}^{1/2} \norm{f}_{H^1(\RR)}^{1/2}
\]
implies that 
\[
\norm{v(t,\cdot)-v_0(\cdot)}_{L^2(\RR)}
\le C \norm{v(t,\cdot)-v_0(\cdot)}_{H^{-1}(\RR)}^{1/2}
\norm{v(t,\cdot)-v_0(\cdot)}_{H^1(\RR)}^{1/2}
\le M(\Etil_0,\ell_0^{-1}) T_0^{1/2}.
\]
Now it follows that 
\[
\norm{v(t,\cdot)-v_0(\cdot)}_{L^\infty(\RR)}
\le C\norm{v(t,\cdot)-v_0(\cdot)}_{L^2(\RR)}^{1/2}
\norm{(v(t,\cdot)-v_0(\cdot))_x}_{L^2(\RR)}^{1/2}
\le M(\Etil_0,\ell_0^{-1}) T_0^{1/4}.
\]
This shows that for \(T_0\) of which smallness depends only on \(\Etil_0\) and \(\ell_0\), 
\[
\inf_{[0,T_0]\times\RR} v(t,x) \ge \frac{\ell_0}{2}.
\]
Moreover, since we have 
\begin{align*}
\norm{\th_t}_{H^{-1}(\RR)}
&\le\k \|\frac{\th_x}{\vtil}\|_{L^2(\RR)}
+R \|\frac{\thtil}{\vtil}\util_x\|_{L^2(\RR)}
+\m \|\frac{(\util_x)^2}{\vtil}\|_{L^2(\RR)} \\
&\le M(\Etil_0,\ell_0^{-1}) + M(\Etil_0,\ell_0^{-1}) \norm{\util_{xx}}_{L^2(\RR)}
\end{align*}
and \(\norm{\util_{xx}}_{L^2(\RR)}^2\) is time-integrable, we obtain 
\[
\norm{\th(t,\cdot)-\th_0(\cdot)}_{H^{-1}(\RR)}
\le \int_0^t \norm{\th_t(s,\cdot)}_{H^{-1}(\RR)} ds
\le M(\Etil_0,\ell_0^{-1}) T_0 + M(\Etil_0,\ell_0^{-1}) T_0^{1/2}.
\]
Then, a similar argument yields the desired conclusion: for \(T_0\) of which smallness depends only on \(\Etil_0\) and \(\ell_0\),
\[
\inf_{[0,T_0]\times\RR} \th(t,x) \ge \frac{\ell_0}{2}.
\]
This completes the proof of the inductive step.

\step{4}{(Convergence)}
In what follows, we show that the iterative sequence is Cauchy in 
\begin{equation}\label{spfc}
\big(L^\infty(0,T;L^2)\cap L^2(0,T;H^1(\RR))\big)^3.
\end{equation}
To this end, we introduce, for each \(n\in \mathbb{N}\), the difference functions
\[
\vbar^n \coloneqq v^n -v^{n-1}, \qquad
\ubar^n \coloneqq u^n -u^{n-1}, \qquad
\thbar^n \coloneqq \th^n-\th^{n-1},
\]
whose time evolutions are given by the following equations:
\begin{equation}\label{vdeq}
\rd_t \vbar^n
=\rd_x \ubar^{n-1}
+\t\Big(\frac{(\vbar^n)_x}{v^{n-1}}\Big)_x
+\t\Big(\big(\frac{1}{v^{n-1}}-\frac{1}{v^{n-2}}\big)(v^{n-1})_x\Big)_x,
\end{equation}
\begin{equation}\begin{aligned}\label{udeq}
\rd_t \ubar^n
&=-R \Big(\th^{n-1}\big(\frac{1}{v^{n-1}}-\frac{1}{v^{n-2}}\big)\Big)_x
-R\Big(\frac{1}{v^{n-2}}(\th^{n-1}-\th^{n-2})\Big)_x\\
&\qquad
+\m\Big(\frac{(\ubar^n)_x}{v^{n-1}}\Big)_x
+\m\Big(\big(\frac{1}{v^{n-1}}-\frac{1}{v^{n-2}}\big)(u^{n-1})_x\Big)_x,
\end{aligned}\end{equation}
and
\begin{equation}\begin{aligned}\label{thdeq}
\frac{R}{\g-1} \rd_t \thbar^n
&=-R\Big(\frac{1}{v^{n-1}}-\frac{1}{v^{n-2}}\Big)\th^{n-1}(u^{n-1})_x
-\frac{R\thbar^{n-1}}{v^{n-2}} (u^{n-1})_x
-\frac{R\th^{n-2}}{v^{n-2}} (\ubar^{n-1})_x\\
&\qquad
+\m\Big(\frac{1}{v^{n-1}}-\frac{1}{v^{n-2}}\Big) ((u^{n-1})_x)^2
+\frac{\m}{v^{n-2}}(u^{n-1}+u^{n-2})_x (\ubar^{n-1})_x\\
&\qquad
+\k\Big(\frac{(\thbar^n)_x}{v^{n-1}}\Big)_x
+\k\Big(\big(\frac{1}{v^{n-1}}-\frac{1}{v^{n-2}}\big)(\th^{n-1})_x\Big)_x.
\end{aligned}\end{equation}
We remark that the notation \(\vbar^n\) for the difference should not be confused with the truncated variable \(\vbar^N\) introduced in Subsection \ref{subsec:DeGiorgi}.
Using these equations, together with the estimates obtained in the previous step, we carry out a standard \(L^2\) energy method as follows:
\begin{align*}
&\rd_t \Big(\frac{1}{2}\int_\RR (\vbar^n)^2 dx\Big)\\
&=-\int_\RR (\vbar^n)_x \ubar^{n-1} dx
-\t \int_\RR \frac{((\vbar^n)_x)^2}{v^{n-1}} dx
-\t \int_\RR (\vbar^n)_x \Big(\frac{1}{v^{n-1}}-\frac{1}{v^{n-2}}\Big) (v^{n-1})_x dx\\
&=-\frac{\t}{2} \int_\RR \frac{((\vbar^n)_x)^2}{v^{n-1}} dx
+C\int_\RR v^{n-1} (\ubar^{n-1})^2 dx
+C\int_\RR \frac{(\vbar^{n-1})^2}{v^{n-1}(v^{n-2})^2} ((v^{n-1})_x)^2 dx.\\
\end{align*}
Thanks to the uniform estimates in the previous step, we obtain
\begin{equation}\begin{aligned}\label{vsum}
&\rd_t \Big(\frac{1}{2}\int_\RR (\vbar^n)^2 dx\Big) + \frac{\t}{2} \int_\RR \frac{((\vbar^n)_x)^2}{v^{n-1}} dx\\
&\le M(\Etil_0)\int_\RR (\ubar^{n-1})^2 dx
+M(\ell_0^{-1}) \|\vbar^{n-1}\|_{L^2(\RR)} \|(\vbar^{n-1})_x\|_{L^2(\RR)}
\int_\RR ((v^{n-1})_x)^2 dx\\
&\le M(\Etil_0,\ell_0^{-1}) \int_\RR \Big((\vbar^{n-1})^2 + (\ubar^{n-1})^2\Big) dx
+\frac{\t}{8}\int_\RR \frac{((\vbar^{n-1})_x)^2}{v^{n-2}} dx.
\end{aligned}\end{equation}
By a similar argument, using the equation for \(u\), we obtain
\begin{align*}
\rd_t \Big(\frac{1}{2}\int_\RR (\ubar^n)^2 dx\Big)
&= R \int_\RR (\ubar^n)_x \Big(\th^{n-1}\big(\frac{1}{v^{n-1}}-\frac{1}{v^{n-2}}\big)
+\frac{\thbar^{n-1}}{v^{n-2}}\Big) dx
-\m \int_\RR \frac{((\ubar^n)_x)^2}{v^{n-1}} dx\\
&\qquad
-\m \int_\RR (\ubar^n)_x \Big(\frac{1}{v^{n-1}}-\frac{1}{v^{n-2}}\Big) (u^{n-1})_x dx\\
&\le -\frac{\m}{2}\int_\RR \frac{((\ubar^n)_x)^2}{v^{n-1}} dx
+C\int_\RR \frac{(\th^{n-1})^2}{v^{n-1}(v^{n-2})^2} (\vbar^{n-1})^2 dx\\
&\qquad
+C\int_\RR \frac{v^{n-1}}{(v^{n-2})^2} (\thbar^{n-1})^2 dx
+C\int_\RR \frac{(\vbar^{n-1})^2}{v^{n-1}(v^{n-2})^2} ((u^{n-1})_x)^2 dx.
\end{align*}
Then, the previous estimates yield the following:
\begin{multline}\label{usum}
\rd_t \Big(\frac{1}{2}\int_\RR (\ubar^n)^2 dx\Big)
+\frac{\m}{2}\int_\RR \frac{((\ubar^n)_x)^2}{v^{n-1}} dx\\
\le M(\Etil_0,\ell_0^{-1}) \int_\RR \Big((\vbar^{n-1})^2 + (\ubar^{n-1})^2 + (\thbar^{n-1})^2\Big) dx
+ \frac{\t}{16}\int_\RR \frac{((\vbar^{n-1})_x)^2}{v^{n-2}} dx.
\end{multline}
Moreover, the temperature difference \(\thbar^n\) satisfies
\begin{align*}
&\rd_t \Big(\frac{R}{2(\g-1)}\int_\RR (\thbar^n)^2 dx\Big)\\
&= -R\int_\RR \Big(\big(\frac{1}{v^{n-1}}-\frac{1}{v^{n-2}}\big)\th^{n-1}
+\frac{\thbar^{n-1}}{v^{n-2}}\Big) \thbar^n (u^{n-1})_x dx
-R\int_\RR \frac{\th^{n-2}}{v^{n-2}} \thbar^n (\ubar^{n-1})_x dx\\
&\qquad
+\m\int_\RR \Big(\frac{1}{v^{n-1}}-\frac{1}{v^{n-2}}\Big) \thbar^n ((u^{n-1})_x)^2 dx
+\m\int_\RR \frac{(u^{n-1}+u^{n-2})_x}{v^{n-2}} \thbar^n (\ubar^{n-1})_x dx\\
&\qquad
-\k \int_\RR \frac{((\thbar^n)_x)^2}{v^{n-1}} dx
-\k \int_\RR (\thbar^n)_x \Big(\frac{1}{v^{n-1}}-\frac{1}{v^{n-2}}\Big) (\th^{n-1})_x dx.
\end{align*}
This leads us to have the following:
\begin{align*}
&\rd_t \Big(\frac{R}{2(\g-1)}\int_\RR (\thbar^n)^2 dx\Big)
\le
-\frac{3}{4}\k \int_\RR \frac{((\thbar^n)_x)^2}{v^{n-1}} dx
+\frac{\m}{4} \int_\RR \frac{((\ubar^{n-1})_x)^2}{v^{n-2}} dx\\
&+ M(\Etil_0,\ell_0^{-1}) \int_\RR \Big((\vbar^{n-1})^2 + (\thbar^{n-1})^2\Big) dx
+M(\Etil_0,\ell_0^{-1}) \int_\RR (\thbar^n)^2 \Big(((\ubar^{n-1})_x)^2 + ((u^{n-1})_x)^2\Big) dx\\
&+ M(\Etil_0,\ell_0^{-1})\int_\RR (\thbar^n)^2 dx
+ M(\Etil_0,\ell_0^{-1})\int_\RR (\vbar^{n-1})^2 \Big(((\th^{n-1})_x)^2 + ((u^{n-1})_x)^2 \Big) dx.
\end{align*}
Then, using 
\[
\|\thbar^n\|_{L^\infty(\RR)}^2
\le C\|\thbar^n\|_{L^2(\RR)} \|(\thbar^n)_x\|_{L^2(\RR)}
\]
and 
\[
\|\vbar^{n-1}\|_{L^\infty(\RR)}^2
\le C\|\vbar^{n-1}\|_{L^2(\RR)} \|(\vbar^{n-1})_x\|_{L^2(\RR)},
\]
we find that
\begin{equation}\begin{aligned}\label{thsum}
&\rd_t \Big(\frac{R}{2(\g-1)}\int_\RR (\thbar^n)^2 dx\Big)
+\frac{\k}{2} \int_\RR \frac{((\thbar^n)_x)^2}{v^{n-1}} dx\\
&\qquad\le
M(\Etil_0,\ell_0^{-1}) \int_\RR \Big((\vbar^{n-1})^2 + (\thbar^{n-1})^2\Big) dx
+M(\Etil_0,\ell_0^{-1}) \int_\RR (\thbar^n)^2 dx\\
&\qquad\qquad
+\frac{\t}{16}\int_\RR \frac{((\vbar^{n-1})_x)^2}{v^{n-2}} dx
+\frac{\m}{4} \int_\RR \frac{((\ubar^{n-1})_x)^2}{v^{n-2}} dx.
\end{aligned}\end{equation}
Combining \eqref{vsum}, \eqref{usum} and \eqref{thsum}, one deduces the following:
\begin{equation}\begin{aligned}\label{bctr}
&\frac{1}{2}\rd_t \int_\RR (\vbar^n)^2 + (\ubar^n)^2 + \frac{R}{\g-1}(\thbar^n)^2 dx
+ \frac{1}{2} \int_\RR \t\frac{((\vbar^n)_x)^2}{v^{n-1}}
+\m\frac{((\ubar^n)_x)^2}{v^{n-1}}
+\k\frac{((\thbar^n)_x)^2}{v^{n-1}} dx\\
&\qquad\le M(\Etil_0,\ell_0^{-1}) \int_\RR (\vbar^{n-1})^2 + (\ubar^{n-1})^2 + (\thbar^{n-1})^2 dx
+ M(\Etil_0,\ell_0^{-1}) \int_\RR (\thbar^n)^2 dx\\
&\qquad\qquad
+\frac{1}{4}\int_\RR \t\frac{((\vbar^{n-1})_x)^2}{v^{n-2}}+\m\frac{((\ubar^{n-1})_x)^2}{v^{n-2}} dx.
\end{aligned}\end{equation}
This, combined with a standard argument, establishes the contraction on a sufficiently small time interval as follows: there exist \(T_0>0\) which depends only on \(\Etil_0\) and \(\ell_0\), and \(0 < \a<1\) such that for any \(n\in \mathbb{N}\), we have 
\[
X^n(T_0) \le \a X^{n-1}(T_0)
\]
where \(X^n(T_0)\) is defined as
\begin{align*}
X^n(T_0) &\coloneqq 
\sup_{t\in[0,T_0]}\int_\RR (\vbar^n)^2 + (\ubar^n)^2 + \frac{R}{\g-1}(\thbar^n)^2 dx\\
&\qquad\qquad
+\int_0^{T_0} \int_\RR \t\frac{((\vbar^n)_x)^2}{v^{n-1}}
+\m\frac{((\ubar^n)_x)^2}{v^{n-1}}
+\k\frac{((\thbar^n)_x)^2}{v^{n-1}} dx dt.
\end{align*}
This shows that the sequence \(\{(v^n,u^n,\th^n)\}\) is Cauchy in the desired space in \eqref{spfc}.

\vspace{2mm}
We now justify convergence, from which the existence of a solution follows immediately.
As shown above, the sequence \(\{(v^n,u^n,\th^n)\}\) is uniformly bounded in 
\begin{equation}\label{sp1}
\big(L^\infty(0,T_0;H^k(\RR)) \cap L^2(0,T_0;H^{k+1}(\RR))\big)^3
\end{equation}
and Cauchy in
\begin{equation}\label{sp2}
\big(L^\infty(0,T_0;L^2(\RR)) \cap L^2(0,T_0;H^1(\RR))\big)^3.
\end{equation}
Since the space \eqref{sp2} is a Banach space, the Cauchy property implies that there exists a limit \((v,u,\th)\) such that 
\[
(v^n,u^n,\th^n) \to (v,u,\th)
\text{ strongly in } \big(L^\infty(0,T_0;L^2(\RR)) \cap L^2(0,T_0;H^1(\RR))\big)^3.
\]
On the other hand, by the uniform boundedness in \eqref{sp1}, there exist a subsequence (not relabeled) and a triple \((v_*,u_*,\th_*)\) such that
\begin{align*}
(v^n,u^n,\th^n) &\rightharpoonup^* (v_*,u_*,\th_*)
\qquad \text{in } \big(L^\infty(0,T_0;H^k(\RR))\big)^3,\\
(v^n,u^n,\th^n) &\rightharpoonup (v_*,u_*,\th_*)
\qquad \text{in } \big(L^2(0,T_0;H^{k+1}(\RR))\big)^3.
\end{align*}
Since both the weak and weak-\(*\) convergence in the higher-order space \eqref{sp1} and the strong convergence in the lower-order space \eqref{sp2} imply convergence in the sense of distributions, the limits must coincide, and hence 
\[
(v,u,\th)
=(v_*,u_*,\th_*)
\in \big(L^\infty(0,T_0;H^k(\RR)) \cap L^2(0,T_0;H^{k+1}(\RR))\big)^3.
\]
Thus, the limit inherits the full regularity of the approximating sequence.
Moreover, the strong convergence in \(L^2(0,T_0;L^2(\RR))\) implies, up to a subsequence, pointwise almost everywhere convergence.
By the Aubin--Lions lemma, we further obtain 
\[
(v^n, u^n, \th^n) \to (v,u,\th)
\text{ in } \big(C([0,T_0];H_{loc}^s(\RR))\big)^3
\]
for any \(s\in(\frac{1}{2},1)\).
In particular, the limit \((v,u,\th)\) admits continuous representatives.
This is sufficient to pass to the limit in the nonlinear terms, so that \((v,u,\th)\) is a solution to \eqref{bnsfsys}.
Finally, combining the almost everywhere convergence with the continuity and the uniform positivity of the approximating sequence, we conclude that 
\[
v(t,x), \th(t,x) \ge \frac{\ell_0}{2}, \qquad \forall (t,x) \in [0,T_0] \times \RR.
\]

\step{6}{(Uniqueness)}
Let \((v_1,u_1,\th_1)\) and \((v_2,u_2,\th_2)\) be two solutions.
Repeating the same estimates in \textit{Step 4} for the differences \((v_1-v_2,u_1-u_2,\th_1-\th_2)\), we obtain the uniqueness.
This completes the proof of the local well-posedness. \qed

\end{appendix}

\subsection*{Conflict of interest}
The authors declare no conflict of interest.

\subsection*{Data availability statement}
The authors do not analyze or generate any datasets, as this work is purely theoretical and mathematical.

\bibliography{reference}

\end{document}

%% file: newmacros.tex
\usepackage{amssymb,mathrsfs,mathtools,graphicx,enumerate,color,colortbl}
\usepackage{hyperref}
\usepackage{comment}

\newcommand{\Ccal}{\mathcal{C}}

\newcommand{\Ecal}{\mathcal{E}}

\newcommand{\Mbar}{\bar{M}}

\newcommand{\Ubar}{\bar{U}}

\renewcommand{\hbar}{\bar{h}}

\newcommand{\ubar}{\bar{u}}
\newcommand{\vbar}{\bar{v}}
\newcommand{\wbar}{\bar{w}}

\newcommand{\thbar}{\bar{\th}}

\newcommand{\util}{\tilde{u}}
\newcommand{\vtil}{\tilde{v}}

\newcommand{\thtil}{\tilde{\th}}

\newcommand{\Etil}{\overline{\Ecal}}

\renewcommand{\a}{\alpha}
\renewcommand{\b}{\beta}
\newcommand{\g}{\gamma}
\renewcommand{\d}{\delta}

\newcommand{\et}{\eta}
\renewcommand{\th}{\theta}
\renewcommand{\k}{\kappa}

\newcommand{\m}{\mu}

\renewcommand{\r}{\rho}

\renewcommand{\t}{\tau}

\newcommand{\RR}{{\mathbb R}}

\newcommand{\abs}[1]{\left|#1\right|}

\newcommand{\norm}[1]{\left\|#1\right\|}

\newcommand{\one}[1]{\mathbf{1}_{\{#1\}}}

\newcommand{\step}[2]{\vskip0.2cm \noindent{\it Step {#1}: {#2}}}

\definecolor{black}{rgb}{0.0, 0.0, 0.0}
\definecolor{red}{rgb}{1.0, 0.5, 0.5}

\newcommand{\red}[1]{\textcolor{red}{#1}}

\newcommand{\rd}{\partial}